\documentclass[12pt,a4paper]{amsart}

\usepackage[margin=2cm]{geometry}
\usepackage{graphicx}
\usepackage{amsmath}
\usepackage{amssymb}
\usepackage{amsthm}
\usepackage{amsfonts}
\usepackage{mathtools}
\usepackage{color}
\usepackage{url}
 
\usepackage{hyperref}
\newcommand{\RR}{\mathbb{R}}

\newcommand{\eps}{\varepsilon}

\newcommand{\dd}{\mathrm{d}}

\newcommand{\BB}{\mathcal{B}}
\newcommand{\TT}{\mathcal{T}}
\newcommand{\MM}{\mathcal{M}}

\newtheorem{theorem}{Theorem}[section]

\newtheorem{proposition}[theorem]{Proposition}
\newtheorem{lemma}[theorem]{Lemma}
\theoremstyle{definition}

\newtheorem{remark}[theorem]{Remark}
\numberwithin{equation}{section}

\usepackage[backend=bibtex,firstinits=true,style=numeric,citestyle=numeric, doi=false,isbn=false,url=false,eprint=false,maxbibnames=99]{biblatex}
\title{Galerkin Approximation of the Fractional Sobolev Constant}

	\author[A.~Dima]{ANDREEA DIMA}
	\address[A.~Dima]{Simion Stoilow Institute of Mathematics of the Romanian Academy, 21 Calea Grivi\c{t}ei Street, 010702, Bucharest, Romania}
\email[]{andreeadima21\@@{}gmail.com,\ adima\@@{}imar.ro}
    
\author[L.~I.~Ignat]{Liviu I. Ignat}
\address[L.~I.~Ignat]{Institute of Mathematics ``Simion Stoilow'' of the Romanian Academy, 21 Calea Grivitei Street, 010702 Bucharest, Romania.
    \newline\indent
 National University of Science and Technology Politehnica Bucharest, 313 Splaiul Independen\c tei, 060042 Bucharest, Romania.    \newline\indent
 Academy of
Romanian Scientists, Ilfov Street, no. 3, Bucharest, Romania.}
\email[]{liviu.ignat\@@{}gmail.com}
\urladdr{https://www.imar.ro/~lignat}

        \subjclass[2020]{65N30, 46E35}
\keywords{Fractional Sobolev Inequalities, Approximation and Stability, Finite Element Method, Gagliardo-Nirenberg Interpolation Inequalities}
\begin{document} 
\begin{abstract}
We establish sharp estimates for the discrete optimal constant of the fractional Sobolev inequality in dimension $N\geq 1$, with fractional exponent $s\in (0,\min\{1,N/2\})$. The convergence rates that we establish take place for the Galerkin approximation with piecewise linear elements, when the problem is posed in the unit ball, for which we employ a quasi-uniform and regular mesh. Since the infimum defining the fractional Sobolev constant is not attained on bounded domains, these estimates quantify how far a piecewise affine mesh of size $h$ can resolve the concentration of the optimizers: the discrete minimizers are close to bubbles concentrated at the scale $h^{2(2-s)/(N+4(1-s))}$, and the energy gap is exactly of order $h^{2(2-s)(N-2s)/(N+4(1-s))}$. 
\end{abstract}

\maketitle
\section{Introduction}

In this paper we investigate the piecewise linear finite element approximation of the fractional Sobolev constant
\begin{equation}
\label{eq:SNs-intro}
S_{N,s}=\inf_{u\in C_c^\infty(\RR^N)\setminus\{0\}} \frac{[u]^2_{\dot{H}^s(\RR^N)}}{\|u\|^2_{L^{2_s^{*}}(\RR^N)}}, 
\end{equation}
for   $N\geq 1$ and $s\in \left(0,\min\{1,\frac{N}{2}\}\right)$, where $$2_s^*=\frac{2N}{N-2s}.$$
The infimum is not attained in
$C_c^\infty(\RR^N)$, but is achieved on two ($N$+2)-dimensional components of the manifold $\mathcal{M}$ \cite{MR717827} (see Section \ref{sec:preliminaries}). The non-compactness of
$\mathcal{M}$, and the associated bubbling phenomenon, is the primary
source of difficulty in both the continuous stability theory and the
discrete approximation problem.

Let us explain what we mean by approximation in this context. The value of $S_{N,s}$ is known in closed form (see Proposition \ref{prop:Fractional-Sobolev-Inequality}), so the purpose of this paper is not its numerical computation. In any case, the constants in Theorem \ref{thm:main} below are not explicit and depend on $N$, $s$ and on the mesh parameters. Our aim is rather to describe the obstruction that a piecewise affine discretisation meets in front of a critical, scale-invariant and non-compact variational problem. On a bounded domain $\Omega$, the infimum of the quotient in \eqref{eq:SNs-intro} over functions vanishing outside $\Omega$ is still $S_{N,s}$, but it is not attained (Proposition \ref{prop:frac-sob-on-domains} and Remark \ref{rem:nonattainment-on-a-bounded-domain}). Minimizing sequences concentrate, that is, they behave like bubbles whose concentration scale tends to zero. As in the local case \cite{pratelli,ignat2025optimalconvergenceratesfinite}, the Galerkin approximation of such an infimum is nevertheless a well-posed question, and the discrete problems always have minimizers. A mesh of size $h$, however, can resolve a bubble only down to a certain scale. Theorem \ref{thm:main} and its proof show that this resolution limit is $c_h\sim h^{\frac{2(2-s)}{N+4(1-s)}}$, and that the resulting energy gap $S_{N,s,h}-S_{N,s}$ is exactly of order $h^{\alpha_{N,s}}$, with $\alpha_{N,s}$ given in \eqref{eq:alpha}. This is a sharp quantitative description of the obstruction met when discretising a critical, scale-invariant problem. Let us also point out that the explicit knowledge of $S_{N,s}$ is precisely what makes the gap $S_{N,s,h}-S_{N,s}$ an accessible object of study.


In the classical (local) setting, the sharp Sobolev inequality
\begin{equation*}
  \|D\phi\|_{L^p(\RR^N)} \geq \Lambda_{N,p}\|\phi\|_{L^{p^*}(\RR^N)},
  \qquad 1\leq p<N,\quad p^*=\frac{Np}{N-p},
\end{equation*}
and the structure of its optimizers are by now well understood, following
foundational works of Talenti, Aubin, and Lieb \cite{MR0463908,MR0448404,MR717827}.
Quantitative stability around the manifold of optimizers, i.e. bounds
on how much a near-minimizer must differ from an exact one, has been
an active line of research \cite{MR1124290,MR4484209}.

On the numerical side, the approximation of optimal Sobolev constants by
finite element methods 
has been studied   over the past two
decades. 
Early work \cite{glowinski} provided numerical evidence\  for
$|\gamma_h-\gamma|=\mathcal{O}(h^2)$ for the high-order Sobolev embedding
on bounded domains in $\RR^d$, $d\leq 3$,
$$\|\phi\|_{L^{\infty}(\Omega)}\leq \gamma \|\phi\|_{H^2(\Omega)\cap H^{1}_{0}(\Omega)},$$
where $\gamma$ is the optimal constant in the above inequality and $\gamma_h$ is its discrete counterpart. 
For the constant $\Lambda_{N,p}$
itself, \cite{pratelli} established the non-sharp two-sided bounds
$C^{-1}h^\lambda\leq\Lambda_{3,2,h}-\Lambda_{3,2}\leq Ch^{1/3}$.
Optimal convergence rates for the piecewise linear approximation of
$\Lambda_{N,p}$ in all dimensions and exponents were subsequently
obtained in \cite{ignat2025optimalconvergenceratesfinite}, relying on
refined stability estimates for Sobolev minimizers \cite{MR4484209}.

In parallel, fractional Sobolev spaces and their associated inequalities
have been extensively studied, including sharp constants, stability, and
improved inequalities (see, e.g., \cite{DiNezzaPalatucci2012,Frank2013,MR4623703,MR4741541}).
From the numerical perspective, finite element methods for fractional Laplacian problems are
now well developed, with convergence and regularity theory established under suitable assumptions \cite{acosta2017fractional}.
Despite these advances, the corresponding problem in the fractional setting remains largely unexplored. 

Let us consider the unit ball $\BB\subset\RR^N$ for which we employ a quasi-uniform and regular mesh of characteristic size $h$. We denote by $V_h$ the corresponding finite-dimensional subspace of $H^1_0(\BB)$, consisting of piecewise linear finite element functions on $\BB$ that are continuous and vanish on the boundary of $\BB$. These functions are extended by zero outside of $\BB$. See Section \ref{sec:discrete-framework} for a precise construction of $V_h$.
Since the fractional Sobolev inequality holds
on any bounded domain with the same sharp constant (Proposition
\ref{prop:frac-sob-on-domains}), the discrete constant
\begin{equation}
\label{eq:SNsh-intro}
S_{N,s,h}\coloneqq\min_{u_h\in V_h\setminus\{0\}}
\frac{[u_h]^2_{\dot{H}^s(\RR^N)}}{\|u_h\|^2_{L^{2_s^*}(\RR^N)}}
\end{equation}
satisfies $S_{N,s,h}\geq S_{N,s}$, and our goal 
  is to provide sharp convergence rates of $S_{N,s,h}$ towards $S_{N,s}$. Note that, in contrast with the continuous problem, the minimum in \eqref{eq:SNsh-intro} is attained, since $V_h$ is finite-dimensional (see Remark \ref{rem:existence-discrete-min}). 
  
\begin{theorem}\label{thm:main}
Let $V_h$ be the space of piecewise linear finite elements, given by a quasi-uniform triangulation of the unit ball $\BB$. For any $s\in \left(0,\min\left\{1,\frac{N}{2}\right\}\right)$, there exist two positive constants $C_1$ and $C_2$ depending only on $N,s$ and the mesh characteristics, such that the following estimate holds for small enough $h$:
\[C_1 h^\alpha\leq S_{N,s,h}-S_{N,s}\leq C_2 h^\alpha,\]
where
\begin{equation}
\label{eq:alpha}
\alpha=\alpha_{N,s}=\frac{2(2-s)(N-2s)}{N+4(1-s)}.
\end{equation}
\end{theorem}

To the best of our knowledge, this is the first result establishing sharp
convergence rates for the finite element approximation of an optimal
constant in a fractional Sobolev inequality.

The exponent $\alpha_{N,s}$ reflects a competition between two error
contributions: the interpolation deficit of a concentrated bubble
$\Phi_{\lambda,c,0}$, of order $(h/c)^{2(2-s)}$, and the energy of its tail outside the unit ball, which has to be cut off in order to satisfy the homogeneous Dirichlet condition, of order $c^{N-2s}$. Balancing these by
choosing $c=c_h\sim h^{\frac{2(2-s)}{N+4(1-s)}}$ yields the rate
$h^\alpha$. The exponent is strictly positive for all  admissible $s$, and
recovers the classical rate $\alpha_{N,1}=\frac{2(N-2)}{N}$ as $s\to 1$,
consistent with \cite{ignat2025optimalconvergenceratesfinite}.

Moreover, the scale $c_h\sim h^{\frac{2(2-s)}{N+4(1-s)}}$ is not only the optimal choice in the upper bound: it is also the scale selected by the discrete problem. The proof of the lower bound shows that any bubble closest to a discrete minimizer concentrates at this scale (see Remark \ref{rem:rate-ch}). In other words, the concentration scale of the bubble closest to a discrete minimizer is of the same order as the optimal scale in the upper bound.
Beyond non-attainment of the infimum, one may ask whether the Euler-Lagrange equation associated with the problem,
\begin{equation}
\label{eq:eulerlagrange}
(-\Delta)^s u=|u|^{2_s^{*}-2} u\ \text{in }\Omega, \ u=0\ \text{in }\RR^N\setminus\Omega
\end{equation}(see Remark \ref{rem:upperbound2} for the definition of the Fractional Laplacian $(-\Delta)^s$)
has any non-trivial solution on the ball; that is, whether the discretisation might be seen as approximating some critical point even in the absence of a minimiser (see Remark \ref{rem:nonattainment-on-a-bounded-domain}). The work of Ros-Oton and Serra (\cite{RosOtonSerra2014Pohozaev}) excludes positive bounded solutions of \eqref{eq:eulerlagrange} when $\Omega$ is bounded, $C^{1,1}$ and star-shaped. Three comments are in order. First, at the critical exponent \cite{RosOtonSerra2014Pohozaev} excludes positive solutions, while the sign-changing case is left open there. For the minimization problem this is immaterial, since the quotient in \eqref{eq:SNs-intro} does not increase when $u$ is replaced by $|u|$. Second, boundedness is not a restriction at the critical exponent, since positive weak solutions of \eqref{eq:eulerlagrange} are bounded by a Brezis--Kato/Moser iteration \cite[Proposition 2.2]{BarriosColoradoServadeiSoria2015}. Third, although the polytopes $\BB_h$ are not covered by \cite{RosOtonSerra2014Pohozaev} (they are not $C^{1,1}$ when $N\geq 2$), non-attainment holds on every bounded domain, by the elementary argument of Remark \ref{rem:nonattainment-on-a-bounded-domain}. Consequently, the discrete problems \eqref{eq:SNsh-intro} have minimizers for every $h$, while the continuous problem on $\BB$ has neither a minimizer nor, as far as the Pohozaev obstruction reaches, a positive critical point. What the mesh sees is the concentration itself.

The proof of Theorem \ref{thm:main} is structured as follows: the upper bound (Section \ref{sec:estimatefromabove}) is obtained by
testing with $I_h(\Psi_{\lambda_h,c_h,0})$, the nodal interpolant of a
truncated bubble $\Psi_{\lambda_h,c_h,0}$ chosen to vanish on $\partial\BB$. Denoting by $\delta(u)$ the \textit{Fractional Sobolev Deficit}, 
the stability estimate  $\delta(u)\sim d(u,\mathcal{M})^2/\|u\|^2_{L^{2_s^*}}$
of \cite{Frank2013} reduces the problem to controlling the interpolation
error in $\dot H^s$.

The lower bound (Section \ref{sec:estimatefrombelow}) is more delicate.
Given a discrete minimizer $u_h$ with small deficit
$\delta(u_h)=S_{N,s,h}-S_{N,s}$, stability forces $u_h$ to be close to
some bubble $\Phi_{\lambda_h,c_h,X_h}$. Steps I--III extract quantitative
information on the parameters $\lambda_h$, $c_h$, $X_h$ (namely, $|\lambda_h|\sim c_h^{-\frac{N-2s}{2}}$, $X_h\in\BB_h$, $c_h\to 0$ and $h/c_h$ stays bounded); Step IV derives the sharp lower bound by balancing the two competing error contributions (see Remark \ref{rem:rate-ch}). Comparing the lower bound of Step IV with the upper bound, we also obtain a posteriori that $c_h\sim h^{\frac{2(2-s)}{N+4(1-s)}}$ (see Remark \ref{rem:rate-ch}).

The strategy is the same as in the local case \cite{ignat2025optimalconvergenceratesfinite}: the upper bound uses the nodal interpolant of a truncated bubble, and the lower bound combines a stability inequality, here the one of Chen, Frank and Weth \cite{Frank2013}, with an analysis of the bubble closest to a discrete minimizer. The nonlocal seminorm requires new arguments in the following steps.
\begin{itemize}
\item \emph{Upper bound (Section \ref{sec:estimatefromabove}).} The energy is a double integral on $\RR^N\times\RR^N$, which we split into $J_1$ (the tail of the bubble outside $\BB_h$), $J_2$ (on $\BB_h\times\BB_h$) and the cross term $J_3$, which has no local counterpart. $J_2$ is estimated by a Gagliardo--Nirenberg inequality on the convex polytopes $\BB_h$, uniform in $h$ (Theorem \ref{thm:IntIneq}, Remark \ref{rem:GN-convex}), and gives the factor $(h/c_h)^{2(2-s)}$.
\item \emph{Step III.} The fractional energy controls $u_h-\Phi_{\lambda_h,c_h,X_h}$ only in $L^{2_s^*}$, not its gradient. We recover the gradient by an interpolation inequality (Lemma \ref{lem:interpolationoncubes}) on a cube of size $\sim c_h$ inside a cone with vertex $X_h$ (a one-sided interval if $N=1$).
\item \emph{Step IV.} The fractional energy does not split over the elements. We keep the interactions inside each element and reduce them, by a fractional Poincar\'e inequality (Theorem \ref{thm:Poincareinequality}), to the $L^2$ oscillation on the element, which is bounded from below by Lemmas \ref{lem:bound-on-gradient} and \ref{lem:interpolationoncubes}. The exterior term $I_{1h}\gtrsim c_h^{N-2s}$ is estimated directly.
\item \emph{Low dimensions.} All $N\geq1$ are covered, whereas the local problem requires $N\geq3$ when $p=2$. For $N=1$, the zeros of the second derivative of the bubble are treated separately (Lemma \ref{lem:finitecovering}).
\end{itemize}

For the reader's convenience, \cite[Lemma 2.2]{ignat2025optimalconvergenceratesfinite}, which is used in Steps III and IV, is reproduced in Lemma \ref{lem:bound-on-gradient}.

Let us finally mention a natural continuation of this work. For Brezis--Nirenberg-type perturbations of the problem, obtained by adding a subcritical or a linear lower-order term to the functional, the continuous problem does have minimizers on bounded domains for suitable ranges of the parameters \cite{BrezisNirenberg1983,ServadeiValdinoci2015}. In that setting, a finite element analysis would approximate an actual solution of an actual equation, and the interplay between the mesh size and the perturbation parameter would be a genuine phenomenon. We believe that the techniques developed here are well suited for this problem.

The paper is organized as follows: Section \ref{sec:preliminaries} presents the continuous and discrete
frameworks. Sections \ref{sec:estimatefromabove} and
\ref{sec:estimatefrombelow} contain the upper and lower bounds,
completing the proof of Theorem \ref{thm:main}.  The Appendix collects
some classical interpolation inequalities  and several technical results that play a crucial role in deriving the desired estimates.

\section{Preliminaries}
\label{sec:preliminaries}
\subsection{The continuous framework}
For $s\in (0,1)$ and a measurable function $u:\RR^N\to\RR$ we define the Fractional Sobolev Seminorm:
$$ [u]^2_{\dot{H}^s(\RR^N)}\coloneqq s(1-s)\int_{\RR^N}\int_{\RR^N} \frac{|u(x)-u(y)|^2}{|x-y|^{N+2s}}\, \dd x\, \dd y\in[0,\infty]$$
and the Homogeneous Fractional Sobolev Space
$\dot H^s(\RR^N)$ as the completion of $C_c^\infty(\RR^N)$ with respect to $[\cdot]_{\dot{H}^s(\RR^N)}$. 

\begin{proposition}[Fractional Sobolev Inequality](\cite[Proposition 3.4]{DiNezzaPalatucci2012} and \cite[Theorem 1.1]{cotsiolis2004})
\label{prop:Fractional-Sobolev-Inequality}
Let $0<s<\min\left\{1,\frac{N}{2}\right\}$. For every $u\in C_c^\infty(\RR^N)$ it holds that:
\[[u]_{\dot{H}^s(\RR^N)}^2\geq S_{N,s} \|u\|^2_{L^{2_s^*}(\RR^N)}, \quad \text{where }2_s^*\coloneqq \frac{2N}{N-2s}.\]
The optimal constant is 
\[S_{N,s}=2 s(1-s)\left( \int_{\RR^N} \frac{1-\cos(\zeta_1)}{|\zeta|^{N+2s}} \dd \zeta\right) 2^{2s} \pi^s \frac{\Gamma \left(\frac{N+2s}{2}\right)}{\Gamma \left(\frac{N-2s}{2}\right)}\left( \frac{\Gamma \left(\frac{N}{2}\right)}{\Gamma \left(N\right)}\right)^{\frac{2s}{N}}\]
and the set of minimizers is the following:
\[\mathcal{M}\coloneqq \left\{\Phi_{\lambda,c,X_0}(x)=\lambda \left(1+\frac{|x-X_0|^2}{c^2}\right)^{-\frac{N-2s}{2}}: \lambda \in \RR\setminus\{0\}, c>0,X_0\in \RR^N\right\}.\]
\end{proposition}
By density, the inequality in Proposition \ref{prop:Fractional-Sobolev-Inequality} holds for every $u\in \dot H^s(\RR^N)$, and equality holds for some $u\in \dot H^s(\RR^N)\setminus\{0\}$ if and only if $u\in\mathcal{M}$ \cite{MR717827}.

For a function $u\in \dot H^s(\RR^N)\setminus\{0\}$ we define the Fractional Sobolev Deficit:
\[\delta(u)\coloneqq \frac{[u]_{\dot{H}^s(\RR^N)}^2}{\|u\|_{L^{2_s^*}(\RR^N)}^2}- S_{N,s}.\]
We will denote by $d(u,\mathcal{M})$ the distance of a function $u\in \dot H^s(\RR^N)$ to the minimizers space $\mathcal{M}$ as follows:
  \[d(u,\mathcal{M})\coloneqq \inf\left\{[u-\Phi]_{\dot{H}^s(\RR^N)}: \Phi\in \mathcal{M}\right\}.\]
In \cite[Lemma 2.2]{konig2023stabilitysobolevinequalityexistence} it is shown that the above infimum is in fact a minimum.
\begin{proposition}{\cite[Theorem 1.1]{Frank2013}}\label{prop:estimate-sobolev-deficit}
  There exist two positive constants $C_1$ and $C_2$, depending only on $N$ and $s\in \left(0,\min\left\{1,\frac{N}{2}\right\}\right)$ such that:
    \[C_1\, \delta(u)\leq \frac{(d(u,\mathcal{M}))^2}{\|u\|_{L^{2_s^*}(\RR^N)}^2}\leq C_2\, \delta(u), \quad \forall\ u\in \dot H^s(\RR^N)\setminus\{0\}.\]
\end{proposition}

\begin{proposition}[Fractional Sobolev Inequality on domains]
\label{prop:frac-sob-on-domains}
Let $N$ be a positive integer, $s\in \left(0,\min\left\{1,\frac{N}{2}\right\}\right)$ and $\Omega\subseteq \RR^N$ be a bounded domain. Then, 
\begin{equation}\label{eq:FracSobolevDomains}
S_{N,s}=\inf_{u\in V_\Omega} \frac{[u]_{\dot{H}^{s}(\RR^N)}^2}{\|u\|_{L^{2_s^*}(\RR^N)}^{2}},
\end{equation}
where $V_{\Omega}\coloneqq \left\{u\in \dot H^s(\RR^N)\setminus \{0\}: u=0\ \text{in}\ \RR^N\setminus\Omega\right\}.$

In particular, if $\Omega$ is the unit ball of $\RR^N$, then, for every $R\in (0,1]$ the following sequence:
\[\Lambda_{\eps}(x)=\left[\left(1+\frac{|x|^2}{\eps^2}\right)^{-\frac{N-2s}{2}}-\left(1+\frac{R^2}{\eps^2}\right)^{-\frac{N-2s}{2}}\right]\chi_{B_R(0)}\]
is a minimizing sequence for \eqref{eq:FracSobolevDomains} as $\eps\to 0$.
\end{proposition}
\begin{proof}
For $\Omega_1\subseteq\Omega_2\subseteq\RR^N$ we have that $V_{\Omega_1}\subseteq V_{\Omega_2}\subseteq \dot H^s(\RR^N)\setminus \{0\}$ which implies 
\begin{equation}\label{eq:sobolev:on-domain-geq-on-full}
    \inf_{u\in V_{\Omega_1}} \frac{[u]_{\dot{H}^{s}(\RR^N)}^2}{\|u\|_{L^{2_s^*}(\RR^N)}^{2}}\geq \inf_{u\in V_{\Omega_2}} \frac{[u]_{\dot{H}^{s}(\RR^N)}^2}{\|u\|_{L^{2_s^*}(\RR^N)}^{2}}\geq S_{N,s}.
\end{equation}

For simplicity, assume $0\in \Omega$. We prove that for any $R>0$ such that $B_R(0)\subset\Omega$
$$\inf_{u\in V_{B_R(0)}} \frac{[u]_{\dot{H}^{s}(\RR^N)}^2}{\|u\|_{L^{2_s^*}(\RR^N)}^{2}}\leq S_{N,s},\ \text{which will give us the desired estimate}.$$

Define, for $\eps>0$, the function 
$$u_{\eps}(x)=\left[\left(1+\frac{|x|^2}{\eps^2}\right)^{-\frac{N-2s}{2}}-\left(1+\frac{R^2}{\eps^2}\right)^{-\frac{N-2s}{2}}\right]\chi_{B_R(0)}.$$
Then $u_\eps=0\ \text{on}\ \RR^N\setminus B_R(0)$ and $u_{\eps}\in C_c(\RR^N)$. We will show that 
\begin{equation}\label{eq:sobolev-limit-u-eps}
    \lim\limits_{\eps\to 0} \frac{[u_{\eps}]_{\dot{H}^{s}(\RR^N)}^2}{\|u_{\eps}\|_{L^{2_s^*}(\RR^N)}^{2}}=S_{N,s}
    \end{equation}
and this will conclude the proof. Indeed, by a change of variables, we get
\begin{equation*}
\begin{split}
\|u_{\eps}\|_{L^{2_s^*}(\RR^N)}^{2_s^*} & =\int_{B_R(0)} \left[\left(1+\frac{|x|^2}{\eps^2}\right)^{-\frac{N-2s}{2}}-\left(1+\frac{R^2}{\eps^2}\right)^{-\frac{N-2s}{2}}\right]^{\frac{2N}{N-2s}} \, \dd x\\
& =\eps^{N} \int_{\RR^N} \chi_{B_{\frac{R}{\eps}}(0)}(z)\left[\left (1+|z|^2\right)^{-\frac{N-2s}{2}}-\left(1+\frac{R^2}{\eps^2}\right)^{-\frac{N-2s}{2}}\right]^{\frac{2N}{N-2s}} \, \dd z.
\end{split} 
\end{equation*}
By Lebesgue's Dominated Convergence Theorem we get that:
\begin{equation}\label{eq:u-eps-L2s-norm}
\lim\limits_{\eps\to 0}\frac{\|u_{\eps}\|_{L^{2_s^*}(\RR^N)}^2}{\eps^{N-2s}}=\| \Phi_{1,1,0}\|_{L^{2_s^*}(\RR^N)}^2
\end{equation}
Next, 
\[
\begin{split}
[u_{\eps}]^2_{\dot{H}^s(\RR^N)}&=s(1-s)\Big(\int_{B_R(0)}\int_{B_R(0)} \frac{|u_{\eps}(x)-u_\eps(y)|^2}{|x-y|^{N+2s}} \, \dd x\, \dd y + 2\int_{B_{R}(0)}\int_{B_R^c(0)} \frac{u_{\eps}^2(x)}{|x-y|^{N+2s}} \,\dd x\, \dd y\Big)\\
&=: s(1-s)(I_1+2 I_2).
\end{split}
\]
By the change of variables $z\coloneqq \frac{x}{\eps}$, $t\coloneqq \frac{y}{\eps}$, we get that 
$$I_1=\eps^{N-2s} \int_{\RR^N}\int_{\RR^N} \chi_{B_{\frac{R}{\eps}}(0)}(z)\, \chi_{B_{\frac{R}{\eps}}(0)}(t)\frac{|\Phi_{1,1,0}(z)-\Phi_{1,1,0}(t)|^2}{|z-t|^{N+2s}}\, \dd z\, \dd t.$$ Again using Lebesgue's Dominated Convergence Theorem, we get that 
\begin{equation}\label{eq:u-eps-I1-converge}
    \lim\limits_{\eps\to 0} \frac{s(1-s)\,I_1}{\eps^{N-2s}}=[\Phi_{1,1,0}]^2_{\dot{H}^s(\RR^N)}.
\end{equation}

It remains to show that $\lim\limits_{\eps\to 0} \frac{I_2}{\eps^{N-2s}}=0$. Indeed,
\begin{equation*}
\begin{split}
I_2&=\int_{B_R(0)}\int_{B^c_R(0)} \frac{\left[\left(1+\frac{|x|^2}{\eps^2}\right)^{-\frac{N-2s}{2}}-\left( 1+\frac{R^2}{\eps^2}\right)^{-\frac{N-2s}{2}}\right]^2}{|x-y|^{N+2s}}\,\dd x \,\dd y \\
&\leq \int_{B_R(0)}\int_{B^c_R(0)} \frac{\left[\left(1+\frac{|x|^2}{\eps^2}\right)^{-\frac{N-2s}{2}}-\left( 1+\frac{|y|^2}{\eps^2}\right)^{-\frac{N-2s}{2}}\right]^2}{|x-y|^{N+2s}}\,\dd x\,\dd y
\end{split}
\end{equation*}
By the same change of variables as above we get that 
\begin{equation}\label{eq:u-eps-I2-convergence}
    \frac{I_2}{\eps^{N-2s}}\leq  \int_{\RR^N}\int_{\RR^N} \chi_{B_{\frac{R}{\eps}}(0)}(z) \,\chi_{B_{\frac{R}{\eps}}^c(0)} (t)\frac{|\Phi_{1,1,0}(z)-\Phi_{1,1,0}(t)|^2}{|z-t|^{N+2s}} \,\dd z \,\dd t \xrightarrow{\eps\to 0} 0.
\end{equation}
The conclusion follows by \eqref{eq:u-eps-L2s-norm}, \eqref{eq:u-eps-I1-converge} and \eqref{eq:u-eps-I2-convergence}.
\end{proof}

\begin{remark}
\label{rem:nonattainment-on-a-bounded-domain}
Note that the infimum in \eqref{eq:FracSobolevDomains} is not attained. Indeed, if $u\in V_{\Omega}$ realised the infimum, then $u$ would be a global optimiser of the Fractional Sobolev Inequality on $\dot{H}^s(\RR^N)$. Therefore, by the classification in Proposition \ref{prop:Fractional-Sobolev-Inequality}, $u=\Phi_{\lambda,c,X_0}\in \mathcal{M}$. But every element of $\mathcal{M}$ is nowhere vanishing on $\RR^N$, contradicting the fact that $u\equiv 0$ in $\RR^N\setminus\Omega$, which is a set of positive (indeed infinite) measure, since $\Omega$ is bounded. This argument holds for every bounded domain $\Omega$ (in particular for $\BB$ and for the polytopes $\BB_h$ of Section \ref{sec:discrete-framework}) and requires no regularity theory. Non-attainment is the feature that drives the analysis of Section \ref{sec:estimatefrombelow}: the discrete minimizers are close to bubbles $\Phi_{\lambda_h,c_h,X_h}$ whose concentration scale $c_h$ tends to zero (Step II), at the rate given in Remark \ref{rem:rate-ch}.
\end{remark}

\subsection{The discrete framework}\label{sec:discrete-framework}
Let $\Omega$ be a polyhedral domain in $\RR^N$. For each small enough positive parameter $h$, we construct a partition $\TT_h$ of the domain $\Omega$ into a finite set of $N-$simplices $T$ satisfying the following properties (\cite[Chapter 2]{ciarlet1978}):
\begin{enumerate}
    \item $\cup_{T\in\TT_h} T=\overline{\Omega};$\\
    \item Each $T\in \TT_h$ is closed and its interior is non-empty;\\
    \item For distinct $T$ and $T'$, their interiors are disjoint;\\
    \item If $T,T'\in \TT_h$, $T\neq T'$, then either $T\cap T'=\emptyset$ or $T\cap T'$ is a common $m-$face, $m\in \{0,\dots,N-1\}$.
\end{enumerate}
For each $T\in  \TT_h$ we denote by $\rho_T$ and $h_T$ the diameter of the largest ball contained in $T$, respectively the diameter of $T$. We denote $$h\coloneqq \max_{T\in \TT_h} h_T.$$ We will consider a set of regular and quasi-uniform meshes $(\TT_h)_{h>0}$ i.e. there exists a positive constant $\sigma$ (independent of $h$) such that $$\frac{h_T}{\rho_T}\leq \sigma$$ for any $T\in\mathcal{T}_h$ and for any $h>0$, respectively $$\inf_{h>0} \frac{\min_{T\in \mathcal{T}_h} h_T}{\max_{T\in\mathcal{T}_h} h_T} \coloneqq \rho>0.$$
Each simplex $T$ of the mesh $\TT_h$ is the image of a reference simplex $\widehat{T}$ (interval when $N=1$/triangle when $N=2$, tetrahedron when $N=3$ etc.) through an affine mapping $$F_T:\RR^N\rightarrow \RR^N,\ F_T(\hat{x})=B_T \hat{x}+b_T,$$ where $B_T$ is an invertible $N\times N$ matrix and $b_T\in \RR^N$ i.e. $$F_T(\widehat{T})=T,\ \forall\ T\in \TT_h.$$ The matrix $B_T$ satisfies: 
\begin{equation}
\label{eq:norm-BT}
\|B_T\|\leq \frac{h_T}{\rho_{\widehat{T}}},\ \|B_T^{-1}\|\leq \frac{h_{\widehat{T}}}{\rho_T}
\end{equation} 
and
\begin{equation}
\label{eq:det-BT}
|\det B_T|=\frac{|T|}{|\widehat{T}|}.
\end{equation}
For a fixed partition $\TT_h$, we define the finite element space 
\[V_h\coloneqq\left\{f\in C(\bar{\Omega}): f\circ F_T\in \mathbb{P}^1(\widehat{T}),\ \forall\ T\in \TT_h,\  f\vert_{\partial\Omega}\equiv 0\right\},\]
where $\mathbb{P}^1(\widehat{T})$ is the space of linear polynomials on $\widehat{T}$.
Let $\BB\coloneqq B_1(0)$ be the unit ball in $\RR^N$. We approximate it with convex polyhedral domains $(\BB_h)_{h>0}$ such that all the boundary nodes of the triangulations $\TT_h$ of $\BB_h$ introduced below lie on $\partial \BB$; in particular, $\BB_h\subset\BB$. In \cite[Section 5.2]{Raviart1983} it is stated that there exists a constant $\theta>0$, independent of $h$, such that 
\begin{equation}
\label{eq:B-Bh}
\BB\setminus \BB_h\subset \{x\in\RR^N:1-\theta h^2\leq |x|<1\}.
\end{equation}
In the sequel, $\TT_h$ is a triangulation of $\Omega=\BB_h$, and $V_h$ is the corresponding finite element space. Extending the functions in $V_h$ by zero outside $\BB_h$, the finite element space $V_h$ can be viewed as a subspace of $H^1_0(\BB)$.

We define the discrete fractional Sobolev constant as:
\begin{equation}
\label{eq:DisccreteFractionalSobolevConstant}
S_{N,s,h}=\min_{u_h \in V_h\setminus\{0\}}\frac{[u_h]_{\dot{H}^s(\RR^N)}^2}{\|u_h\|_{L^{2_s^*}(\RR^N)}^2}.
\end{equation}

\begin{remark}
\label{rem:existence-discrete-min}
We want to point out that the above definition is correct (i.e. the right-hand side of \eqref{eq:DisccreteFractionalSobolevConstant} is indeed a minimum), since the Fractional Sobolev Seminorm of an extension by zero of a function in $V_h$ is finite, the quantity in the right-hand side of the above definition is homogeneous of degree 0 and the unit sphere in the finite dimensional space $V_h$ induced by the $L^{2_s^*}$-norm is compact.
\end{remark}

We introduce the following notation: for two expressions $E$ and $F$, we write $E\lesssim F$ if there exists a constant $C>0$ which depends only on the dimension $N$, the fractional exponent $s\in (0,1)$, and the constants $\sigma$ and $\rho$ such that $E\leq C\, F$. We also write $E\sim F$ provided that both $E\lesssim F$ and $F\lesssim E$ hold true. If the constants involved also depend on some other parameters e.g. $p$, $q$, then we write $\lesssim_p$, $\lesssim_{p,q}$, $\sim_q$ etc. Moreover, if the expression $E$ depends on $h$ and $c$ is a positive constant which may depend only on $N$, $s$, $\sigma$ and $\rho$, we will write $E\simeq c$ if $\lim\limits_{h\to 0} |E-c|=0$. We point out that $\rho$ denotes the quasi-uniformity constant of the family of meshes, while $\rho_T$ denotes the diameter of the largest ball contained in $T$. The constants $C_0$, $C_1$, $C_2$ appearing in Section \ref{sec:estimatefrombelow} are not related to the constants $C_1$, $C_2$ in Theorem \ref{thm:main} and Proposition \ref{prop:estimate-sobolev-deficit}.

\section{Upper Bound in Theorem \ref{thm:main}}
\label{sec:estimatefromabove}

The main tool that we use in order to prove Theorem \ref{thm:main} is the estimate on the deficit in Proposition \ref{prop:estimate-sobolev-deficit}. Thus, in order to prove the second inequality in Theorem \ref{thm:main}, it suffices to find a family of functions in $V_h$ such that their Sobolev deficit is of order $\mathcal{O}(h^\alpha)$. In the following, we construct such a family.

For $\lambda,c>0$, let us define $\Psi_{\lambda,c,0}:\RR^N\rightarrow\RR $ to be a perturbation of $\Phi_{\lambda,c,0}$ constructed in order to vanish on the boundary of $\BB$:
\begin{equation}
\label{def.psi}
  \Psi_{\lambda,c,0}\coloneqq\Phi_{\lambda,c,0}-\lambda\left(1+\frac{1}{c^2}\right)^{-\frac{N-2s}{2}}.
\end{equation}
 In particular it satisfies 
\begin{equation}
\label{ineg.psi}
  0\leq \Psi _{\lambda,c,0}(x)\leq  \Phi_{\lambda,c,0}(x), \quad x\in  \overline{\BB}. 
\end{equation}

Throughout the section $c_h$ is a small parameter such that  $h<< c_h<<1 $. We will choose $\lambda_h$ such that $\|\Psi_{\lambda_h,c_h,0}\|_{L^{2_s^*}(\BB)}=1$. In view of Lemma \ref{lem:estimateforlambda}, we must have 
\begin{equation}
\label{rel.lambdah}
  |\lambda_h|\sim c_h^{-\frac{N-2s}{2}}.
\end{equation}
We also denote by $I_h(\Psi_{\lambda_h,c_h,0})\in V_h$ the piecewise linear interpolant of $\Psi_{\lambda_h,c_h,0}$ extended with zero outside $\BB_h$. Note that, since all the boundary nodes of $\TT_h$ lie on $\partial \BB$ and $\Psi_{\lambda_h,c_h,0}$ vanishes on $\partial\BB$, the interpolant $I_h(\Psi_{\lambda_h,c_h,0})$ vanishes on $\partial \BB_h$, so that $I_h(\Psi_{\lambda_h,c_h,0})\in V_h$.

Along the paper we will denote by $D^2 u$ the  Hessian matrix associated with $u$ and  
\[
|D^2u(x)|=\left(\sum _{i,j=1}^N |\partial_{x_ix_j}u(x)|^2 \right)^{1/2}.
\]

Using the definition of $\Psi_{\lambda,c,0}$ in \eqref{def.psi} and  Proposition \ref{prop:estimate-sobolev-deficit} we get:
\begin{equation}\label{eq:Ineq-with-distance}
\begin{split}
S_{N,s,h}-S_{N,s}&\leq \delta(I_h(\Psi_{\lambda_h,c_h,0}))\\
& \lesssim \frac{(d(I_h(\Psi_{\lambda_h,c_h,0}),\MM))^2}{\|I_h(\Psi_{\lambda_h,c_h,0})\|^2_{L^{2_s^*}(\RR^N)}}\\
&\lesssim \frac{[I_h(\Psi_{\lambda_h,c_h,0})-\Phi_{\lambda_h,c_h,0}]^2_{\dot{H}^s(\RR^N)}}{\|I_h(\Psi_{\lambda_h,c_h,0})\|^2_{L^{2_s^*}(\RR^N)}}\\
&\lesssim \frac{[I_h(\Psi_{\lambda_h,c_h,0})-\Psi_{\lambda_h,c_h,0}]^2_{\dot{H}^s(\RR^N)}}{\|I_h(\Psi_{\lambda_h,c_h,0})\|^2_{L^{2_s^*}(\RR^N)}}+\frac{[\Psi_{\lambda_h,c_h,0}-\Phi_{\lambda_h,c_h,0}]^2_{\dot{H}^s(\RR^N)}}{\|I_h(\Psi_{\lambda_h,c_h,0})\|^2_{L^{2_s^*}(\RR^N)}}\\
&=\frac{[I_h(\Psi_{\lambda_h,c_h,0})-\Psi_{\lambda_h,c_h,0}]^2_{\dot{H}^s(\RR^N)}}{\|I_h(\Psi_{\lambda_h,c_h,0})\|^2_{L^{2_s^*}(\RR^N)}}.
\end{split}
\end{equation}

{\begin{remark}
\label{rem:upperbound2}
The inequality
$$S_{N,s,h}-S_{N,s}\leq \frac{[I_h(\Psi_{\lambda_h,c_h,0})-\Phi_{\lambda_h,c_h,0}]^2_{\dot{H}^s(\RR^N)}}{\|I_h(\Psi_{\lambda_h,c_h,0})\|^2_{L^{2_s^*}(\RR^N)}}$$
can also be obtained directly, without using Proposition \ref{prop:estimate-sobolev-deficit}. Write $v\coloneqq I_h(\Psi_{\lambda_h,c_h,0})$, $\Phi\coloneqq\Phi_{\lambda_h,c_h,0}$, $w\coloneqq v-\Phi$, $\|\cdot\|\coloneqq\|\cdot\|_{L^{2_s^*}(\RR^N)}$, and denote by
$$\langle u,\varphi\rangle\coloneqq s(1-s)\int_{\RR^N}\int_{\RR^N} \frac{(u(x)-u(y))(\varphi(x)-\varphi(y))}{|x-y|^{N+2s}} \, \dd x\, \dd y$$
the bilinear form associated with $[\cdot]^2_{\dot H^s(\RR^N)}$. Since $\Phi$ minimizes the quotient in Proposition \ref{prop:Fractional-Sobolev-Inequality}, it satisfies the Euler--Lagrange equation
\begin{equation}
\label{eq:EL-Phi}
\langle \Phi,\varphi\rangle=S_{N,s}\|\Phi\|^{2-2_s^*}\int_{\RR^N}|\Phi|^{2_s^*-2}\Phi\,\varphi\,\dd x\qquad\text{for all }\varphi\in\dot H^s(\RR^N).
\end{equation}
Equivalently, $(-\Delta)^s\Phi=\frac{S_{N,s}}{2s(1-s)}\|\Phi\|^{2-2_s^*}|\Phi|^{2_s^*-2}\Phi$, where
$$(-\Delta)^sv(x)\coloneqq p.v.\int_{\RR^N} \frac{v(x)-v(y)}{|x-y|^{N+2s}} \, \dd y, \quad x\in\RR^N,$$
since $\langle u,\varphi\rangle=2s(1-s)\int_{\RR^N}\varphi\,(-\Delta)^s u\,\dd x$.
Taking $\varphi=w$ in \eqref{eq:EL-Phi} and using $[\Phi]^2_{\dot H^s(\RR^N)}=S_{N,s}\|\Phi\|^2$ and H\"older's inequality, we get
\begin{align*}
[\Phi]^2_{\dot H^s(\RR^N)}+2\langle\Phi,w\rangle&=-S_{N,s}\|\Phi\|^2+2S_{N,s}\|\Phi\|^{2-2_s^*}\int_{\RR^N}|\Phi|^{2_s^*-2}\Phi\, v\,\dd x\\
&\leq -S_{N,s}\|\Phi\|^2+2S_{N,s}\|\Phi\|\,\|v\|=S_{N,s}\|v\|^2-S_{N,s}\big(\|\Phi\|-\|v\|\big)^2\leq S_{N,s}\|v\|^2.
\end{align*}
Therefore $[v]^2_{\dot H^s(\RR^N)}=[\Phi]^2_{\dot H^s(\RR^N)}+2\langle\Phi,w\rangle+[w]^2_{\dot H^s(\RR^N)}\leq S_{N,s}\|v\|^2+[w]^2_{\dot H^s(\RR^N)}$, and
$$S_{N,s,h}-S_{N,s}\leq \frac{[v]^2_{\dot H^s(\RR^N)}}{\|v\|^2}-S_{N,s}\leq \frac{[w]^2_{\dot H^s(\RR^N)}}{\|v\|^2}.$$
\end{remark}
}
In the following we will prove that the last term in \eqref{eq:Ineq-with-distance} is of order $\mathcal{O}(h^{\alpha})$. 
The following inequalities will play an important role:
\begin{lemma}{\cite[Lemma 2.1]{ignat2025optimalconvergenceratesfinite}}
\label{lem:interpolating-W2infty-functions}
Let $\TT_h$ be a regular mesh on a polyhedral domain $\Omega\subset \RR^N,\ N\geq 1$ and $k\in\{0,1\}$. Then, for all $1\leq p<\infty$ and $|\beta|=k$, there exists a positive constant $C=C_{N,p,k,\sigma}$ such that:
\begin{equation*}
\left(\sum_{T\in \TT_h} \|D^\beta(u-I_h u)\|_{L^p(T)}^p\right)^{\frac{1}{p}}\leq Ch^{2-k} \left(\sum_{T\in\TT_h} |T|\|D^2 u\|_{L^{\infty}(T)}^p\right)^{\frac{1}{p}}, \quad\text{for any}\ u\in W^{2,\infty}(\Omega),
\end{equation*}
and
\begin{equation*}
\max_{T\in \TT_h} \|D^\beta(u-I_h u)\|_{L^{\infty}(T)}\leq C h^{2-k} \max_{T\in\TT_h} \|D^2 u\|_{L^{\infty}(T)}, \quad\text{for any}\ u\in W^{2,\infty}(\Omega).
\end{equation*}
Moreover, for any $T\in\TT_h$ and $p\geq 1$:
\begin{equation*}
\|D^\beta I_h u\|_{L^p(T)}\lesssim h_T^{-k+\frac{N}{p}} \|u\|_{L^{\infty}(T)}, \quad\text{for any}\ u\in C^{1}(T).
\end{equation*}
\end{lemma}

\subsection{Proof of the upper bound.}

We will estimate the seminorm $[I_h(\Psi_{\lambda_h,c_h,0})-\Psi_{\lambda_h,c_h,0}]^2_{\dot{H}^s(\mathbb{R}^N)}$ by splitting it in the following three integrals:

\begin{equation}
\label{eq:integral-on-complement}
J_1\coloneqq \int_{\BB_h^c}\int_{\BB_h^c} \frac{|\Psi_{\lambda_h,c_h,0}(x)-\Psi_{\lambda_h,c_h,0}(y)|^2}{|x-y|^{N+2s}} \, \dd x \, \dd y,
\end{equation}

\begin{equation}
\label{eq:integral-in}
J_2\coloneqq  \int_{\BB_h} \int_{\BB_h} \frac{|(I_h(\Psi_{\lambda_h,c_h,0})-\Psi_{\lambda_h,c_h,0})(x)-(I_h(\Psi_{\lambda_h,c_h,0})-\Psi_{\lambda_h,c_h,0})(y)|^2}{|x-y|^{N+2s}} \, \dd x\, \dd y
\end{equation}

and

\begin{equation}
\label{eq:mixt-integral} 
J_3\coloneqq 2\int_{\BB_h^c} \int_{\BB_h} \frac{|(I_h(\Psi_{\lambda_h,c_h,0})-\Psi_{\lambda_h,c_h,0})(x)-(I_h(\Psi_{\lambda_h,c_h,0})-\Psi_{\lambda_h,c_h,0})(y)|^2}{|x-y|^{N+2s}} \, \dd x \, \dd y.
\end{equation}

We claim the following estimates 
\begin{equation}\label{claim.1}
J_1\lesssim c_h^{N-2s},
\end{equation}
\begin{equation}
\label{claim.2}
   J_2\lesssim   \left(\frac h{c_h}\right)^{2(2-s)},
\end{equation}
\begin{equation}
\label{claim.3}
   J_3\lesssim   c_h^{N-2s}+\left(\frac h{c_h}\right)^{2(2-s)},
\end{equation}
and
\begin{equation}
\label{claim.4}
   \|I_h(\Psi_{\lambda_h,c_h,0})\|_{L^{2_s^*}(\RR^N)}\simeq 1.\end{equation}

Putting all together it gives us 
\[\frac{[I_h(\Psi_{\lambda_h,c_h,0})-\Psi_{\lambda_h,c_h,0}]^2_{\dot{H}^s(\mathbb{R}^N)}}{ \|I_h(\Psi_{\lambda_h,c_h,0})\|^2_{L^{2_s^*}(\RR^N)}}\lesssim c_h^{N-2s}+\left(\frac h {c_h}\right)^{2(2-s)}.
\]
Taking $c_h>>h$ such that the two terms in the right hand side to be equal we obtain the desired estimate.

\vspace{0.5cm}

In the following we will prove the above claims.
 
\begin{lemma}
\label{lem:J1} The term $J_1$ satisfies
\begin{equation}
J_1\lesssim c_h^{N-2s}.
\end{equation}
\end{lemma}
\begin{proof}
We have the following bound for $J_1$: 
\begin{equation*}
	\begin{split}
	J_1 & \leq\int_{|x|\geq \frac{1}{2}} \int_{|y|\geq \frac{1}{2},\ |x-y|\leq \frac{|x|}{2}} \frac{|\Phi_{\lambda_h,c_h,0}(x)-\Phi_{\lambda_h,c_h,0}(y)|^2}{|x-y|^{N+2s}} \, \dd x \, \dd y\\ 
	& + \int_{|x|\geq \frac{1}{2}} \int_{|y|\geq \frac{1}{2},\ |x-y|>\frac{|x|}{2}} \frac{|\Phi_{\lambda_h,c_h,0}(x)-\Phi_{\lambda_h,c_h,0}(y)|^2}{|x-y|^{N+2s}} \, \dd x \, \dd y\\
&=: J_{1,1}+J_{1,2}\\
	\end{split}
\end{equation*}

To estimate $J_{1,1}$ we proceed as follows:
\begin{equation*}
J_{1,1}\leq \int_{|x|\geq \frac{1}{2}} \int_{|y|\geq \frac{1}{2},\ |x-y|\leq \frac{|x|}{2}} \frac{|D\Phi_{\lambda_h,c_h,0}(\xi)|^2}{|x-y|^{N+2s-2}} \, \dd x\, \dd y, 
\end{equation*} 
where $\xi=(1-t)x+ty$ for some $t\in [0,1]$. Since $|x-y|\leq\frac{|x|}{2}$, we get that $|\xi|=|x-t(x-y)|\geq |x|-t|x-y|\geq \frac{|x|}{2}$. Using Lemma $\ref{lem:GradHessMinimizers}$ and the fact that $\frac{|x|}{2}\geq \frac{1}{4}$, we obtain 
\begin{equation*}
\begin{split}
|D\Phi_{\lambda_h,c_h,0}(\xi)|^2\leq \left|D\Phi_{\lambda_h,c_h,0}\left(\frac{x}{2}\right)\right|^2&\sim \frac{|\lambda_h|^2}{c^4_h} \frac{c^{2(N-2s+2)}_h}{\left(c^2_h+\frac{|x|^2}{4}\right)^{N-2s+2}} |x|^2\\
&\sim c^{N-2s}_h |x|^{-(2N-4s+2)}.\\
\end{split}
\end{equation*}
Using that $\int_{|x-y|\leq \frac{|x|}{2}} \frac{1}{|x-y|^{N+2s-2}} \, \dd y\sim |x|^{2-2s}$ we get:
\begin{equation}
\label{eq:J11}
	\begin{split}
	J_{1,1} & \leq \int_{|x|\geq \frac{1}{2}} \left|D\Phi_{\lambda_h,c_h,0}\left(\frac{x}{2}\right)\right|^2 \,  \int_{|x-y|\leq\frac{|x|}{2}} \frac{1}{|x-y|^{N+2s-2}} \, \dd y\, \dd x\\
	& \lesssim c^{N-2s}_h\int_{|x|\geq \frac{1}{2}} |x|^{-2N+2s}\, \dd x\lesssim c^{N-2s}_h.
	\end{split}
\end{equation}

For $J_{1,2}$ we get:\\
\begin{equation}
\label{eq:J12-2-pieces}
\begin{split}
J_{1,2} & \lesssim \int_{|x|\geq \frac{1}{2}} \int_{|y|\geq \frac{1}{2},\ |x-y|> \frac{|x|}{2}} \frac{|\Phi_{\lambda_h,c_h,0}(x)|^2}{|x-y|^{N+2s}} \, \dd x \, \dd y\\
	& + \int_{|x|\geq \frac{1}{2}} \int_{|y|\geq \frac{1}{2},\ |x-y|> \frac{|x|}{2}} \frac{|\Phi_{\lambda_h,c_h,0}(y)|^2}{|x-y|^{N+2s}} \, \dd x \, \dd y.
\end{split}
\end{equation}
Using that $\int_{|z|>\frac{|x|}{2}} \frac{1}{|z|^{N+2s}} \, \dd z\sim |x|^{-2s}$, for the first integral in \eqref{eq:J12-2-pieces} we get:
\begin{equation*}
\begin{split}
\int_{|x|\geq \frac{1}{2}} \int_{|y|\geq \frac{1}{2},\ |x-y|> \frac{|x|}{2}} \frac{|\Phi_{\lambda_h,c_h,0}(x)|^2}{|x-y|^{N+2s}} \, \dd x \, \dd y&\lesssim \int_{|x|\geq \frac{1}{2}} \frac{|\Phi_{\lambda_h,c_h,0}(x)|^2}{|x|^{2s}} \, \dd x\\
&= |\lambda_h|^2 c^{2(N-2s)}_h \int_{|x|\geq \frac{1}{2}} \frac{1}{(c^2_h+|x|^2)^{N-2s}} \frac{1}{|x|^{2s}} \, \dd x\\
&\sim c^{N-2s}_h \int_{|x|\geq \frac{1}{2}} \frac{1}{|x|^{2(N-s)}} \, \dd x\sim c^{N-2s}_h.
\end{split}
\end{equation*}
We split the second integral in \eqref{eq:J12-2-pieces} in two pieces and we get:
\begin{equation}
\label{eq:J12-2}
\begin{split}
\int_{|x|\geq \frac{1}{2}} \int_{|y|\geq \frac{1}{2},\ |x-y|> \frac{|x|}{2}} \frac{|\Phi_{\lambda_h,c_h,0}(y)|^2}{|x-y|^{N+2s}} \, \dd x \, \dd y&=\int_{|x|\geq \frac{1}{2}} \int_{|y|\geq \frac{|x|}{2},\ |y|\geq \frac{1}{2},\ |x-y|> \frac{|x|}{2}} \frac{|\Phi_{\lambda_h,c_h,0}(y)|^2}{|x-y|^{N+2s}} \, \dd x \, \dd y\\
&+\int_{|x|\geq \frac{1}{2}} \int_{\frac{1}{2}\leq |y|< \frac{|x|}{2},\ |x-y|> \frac{|x|}{2}} \frac{|\Phi_{\lambda_h,c_h,0}(y)|^2}{|x-y|^{N+2s}} \, \dd x \, \dd y.
\end{split}
\end{equation}
If $|y|\geq \frac{|x|}{2}$, we get that $|\Phi_{\lambda_h,c_h,0}(y)|^2\leq \left|\Phi_{\lambda_h,c_h,0}\left(\frac{x}{2}\right)\right|^2$, while if $|y|<\frac{|x|}{2}$, we get $|x-y|> \frac{|x|}{2}>|y|$ and, in either case, both integrals in the right hand-side of \eqref{eq:J12-2} are estimated as the first integral in \eqref{eq:J12-2-pieces}. Thus, we get that
\begin{equation}
\label{eq:J12}
J_{1,2}\lesssim c_h^{N-2s}.
\end{equation}  
The result now follows from \eqref{eq:J11} and \eqref{eq:J12}.
\end{proof}

\begin{lemma}
\label{lem:J2} 
The term $J_2$ satisfies
 $$   J_2\lesssim   \left(\frac h{c_h}\right)^{2(2-s)}.$$
\end{lemma}
\begin{proof}
We use the Homogeneous Gagliardo-Nirenberg Interpolation Inequality \ref{thm:IntIneq} with $\Omega=\BB_h,\ k=p=q=2,\ s_0=0,\ s_1=1,\ \theta=s$, whose constant does not depend on $h$ since $\BB_h$ is convex (see Remark \ref{rem:GN-convex}), and we obtain:
\begin{equation}
\label{eq:Int-J2}
J_2\lesssim \|I_h(\Psi_{\lambda_h,c_h,0})-\Psi_{\lambda_h,c_h,0}\|^{2(1-s)}_{L^2(\BB_h)} \|D(I_h(\Psi_{\lambda_h,c_h,0})-\Psi_{\lambda_h,c_h,0})\|^{2s}_{L^2(\BB_h)}.
\end{equation}
Using Lemma \ref{lem:interpolatior-on-polyhedron-Lq-error} for $q=2$ we get that $\|I_h(\Psi_{\lambda_h,c_h,0})-\Psi_{\lambda_h,c_h,0}\|^{2}_{L^2(\BB_h)}\lesssim h^4 c_h^{-2(2-s)}$. 
Using Lemma \ref{lem:gradient-of-interpolatior-error-on-polyhedron} for $p=2$ we get that $\|D(I_h(\Psi_{\lambda_h,c_h,0})-\Psi_{\lambda_h,c_h,0})\|^{2}_{L^2(\BB_h)}\lesssim h^2 c_h^{-2(2-s)}$.
The result now follows.
\end{proof}

We now estimate $J_3$:

\begin{lemma}
\label{lem:J3}
The term $J_3$ satisfies the following estimate: $$    J_3\lesssim   c_h^{N-2s}+\left(\frac h{c_h}\right)^{2(2-s)}.$$
\end{lemma}
\begin{proof}
Since $\Psi_{\lambda_h,c_h,0}-\Phi_{\lambda_h,c_h,0}$ is constant, the functions $I_h(\Psi_{\lambda_h,c_h,0})-\Psi_{\lambda_h,c_h,0}$ and $I_h(\Psi_{\lambda_h,c_h,0})-\Phi_{\lambda_h,c_h,0}$ differ by a constant. Hence, in the integrand of $J_3$ we may replace $I_h(\Psi_{\lambda_h,c_h,0})-\Psi_{\lambda_h,c_h,0}$ by $I_h(\Psi_{\lambda_h,c_h,0})-\Phi_{\lambda_h,c_h,0}$, which equals $-\Phi_{\lambda_h,c_h,0}$ in $\BB_h^c$. We emphasize that in the estimates below we have to work with $\Phi_{\lambda_h,c_h,0}$ rather than with $\Psi_{\lambda_h,c_h,0}$ outside $\BB_h$: inequality \eqref{ineg.psi} holds only in $\overline{\BB}$, and $|\Psi_{\lambda_h,c_h,0}(x)|$ tends to the positive constant $|\lambda_h|(1+c_h^{-2})^{-\frac{N-2s}{2}}$ as $|x|\to\infty$.
We bound $J_3$ with the following two terms:
\begin{equation*}
J_{3,1}\coloneqq \int_{\BB^c_h} \int_{|y|<1,\ |x-y|\geq \frac{|x|}{2}} \frac{|-\Phi_{\lambda_h,c_h,0}(x)-(I_h(\Psi_{\lambda_h,c_h,0})(y)-\Phi_{\lambda_h,c_h,0}(y))|^2}{|x-y|^{N+2s}} \, \dd x\, \dd y
\end{equation*}
and
\begin{equation*}
J_{3,2}\coloneqq \int_{\frac{2}{3}\leq |x|}\int_{|y|<1,\ |x-y|<\frac{|x|}{2}} \frac{|(I_h(\Psi_{\lambda_h,c_h,0})(x)-\Psi_{\lambda_h,c_h,0}(x))-(I_h(\Psi_{\lambda_h,c_h,0})(y)-\Psi_{\lambda_h,c_h,0}(y))|^2}{|x-y|^{N+2s}} \, \dd x\, \dd y,
\end{equation*}
where we used that $I_h(\Psi_{\lambda_h,c_h,0})$ vanishes outside $\BB_h$.
For $J_{3,1}$ we have the following estimate:
\begin{equation*}
\begin{split}
J_{3,1} & \lesssim \int_{\BB^c_h} |\Phi_{\lambda_h,c_h,0}(x)|^2  \dd x \int_{|y|<1,\ |x-y|\geq \frac{|x|}{2}} \frac{1}{|x-y|^{N+2s}}  \dd y\\
& + \int_{|y|<1} |I_h(\Psi_{\lambda_h,c_h,0})(y)-\Phi_{\lambda_h,c_h,0}(y)|^2 \, \dd y \int_{|x-y|\geq\frac{1}{3}} \frac{1}{|x-y|^{N+2s}}  \dd x\\
&\lesssim \int_{|x|\geq \frac{2}{3}} \frac{|\Phi_{\lambda_h,c_h,0}(x)|^2}{|x|^{2s}}  \dd x +\int_{\BB} |I_h(\Psi_{\lambda_h,c_h,0})(y)-\Phi_{\lambda_h,c_h,0}(y)|^2 \dd y,
\end{split}
\end{equation*}
since $\int_{|z|\geq \frac{|x|}{2}} \frac{1}{|z|^{N+2s}} \, \dd z\sim |x|^{-2s}$ and $\int_{|z|\geq\frac{1}{3}} \frac{1}{|z|^{N+2s}}\, \dd z<\infty$.
Using the explicit form of $\Phi_{\lambda_h,c_h,0}$ and \eqref{rel.lambdah}, we obtain
\begin{equation*}
\begin{split}
 \int_{|x|\geq \frac{2}{3}} \frac{|\Phi_{\lambda_h,c_h,0}(x)|^2}{|x|^{2s}}  \dd x & \lesssim |\lambda_h|^2 \int_{\frac{2}{3}}^{\infty} \sigma^{N-2s-1} \left(1+\frac{\sigma^2}{c^2_h}\right)^{-(N-2s)}\, \dd \sigma\\
&\lesssim c_h^{N-2s}\int_{\frac{2}{3}}^{\infty} \sigma^{-1-(N-2s)} \, \dd \sigma\\
&\lesssim c_h^{N-2s}.
\end{split}
\end{equation*}
For the second term we write
\[
\int_{\BB} |I_h(\Psi_{\lambda_h,c_h,0})-\Phi_{\lambda_h,c_h,0}|^2 \dd y\leq 2\int_{\BB} |I_h(\Psi_{\lambda_h,c_h,0})-\Psi_{\lambda_h,c_h,0}|^2 \dd y+2\int_{\BB} |\Psi_{\lambda_h,c_h,0}-\Phi_{\lambda_h,c_h,0}|^2 \dd y.
\]
Since $\Psi_{\lambda_h,c_h,0}-\Phi_{\lambda_h,c_h,0}\equiv-\lambda_h\left(1+\frac{1}{c_h^2}\right)^{-\frac{N-2s}{2}}$ and $\left(1+\frac{1}{c_h^2}\right)^{-1}\leq c_h^2$, by \eqref{rel.lambdah} we have
\[
\int_{\BB} |\Psi_{\lambda_h,c_h,0}-\Phi_{\lambda_h,c_h,0}|^2 \dd y=|\BB|\,|\lambda_h|^2\left(1+\frac{1}{c_h^2}\right)^{-(N-2s)}\leq |\BB|\,|\lambda_h|^2 c_h^{2(N-2s)}\lesssim c_h^{N-2s}.
\]
For the remaining term, inequality \eqref{ineg.psi} can be used, since it is only needed in $\overline{\BB}$. Using \eqref{eq:B-Bh},
we also get:
\begin{equation*}
\begin{split}
\int_{\BB} &|I_h(\Psi_{\lambda_h,c_h,0})(y)-\Psi_{\lambda_h,c_h,0}(y)|^2 \dd y\\
&\lesssim \int_{\BB_h} |I_h(\Psi_{\lambda_h,c_h,0})(y)-\Psi_{\lambda_h,c_h,0}(y)|^2 \, \dd y+\int_{1-\theta h^2\leq |y|\leq 1} |\Psi_{\lambda_h,c_h,0}(y)|^2\, \dd y\\
&\lesssim \int_{\BB_h} |I_h(\Psi_{\lambda_h,c_h,0})(y)-\Psi_{\lambda_h,c_h,0}(y)|^2 \, \dd y+|\lambda_h|^2\int_{1-\theta h^2}^{1} \sigma^{N-1}\left(1+\frac{\sigma^2}{c^2_h}\right)^{-(N-2s)} \, \dd \sigma\\
&\lesssim \int_{\BB_h} |I_h(\Psi_{\lambda_h,c_h,0})(y)-\Psi_{\lambda_h,c_h,0}(y)|^2 \, \dd y+ c^{N-2s}_h. 
\end{split}
\end{equation*}
Using Lemma \ref{lem:interpolatior-on-polyhedron-Lq-error} for $q=2$ we get that 
$$\int_{\BB_h} |I_h(\Psi_{\lambda_h,c_h,0})(y)-\Psi_{\lambda_h,c_h,0}(y)|^2 \, \dd y\lesssim h^4 c_h^{-2(2-s)}\lesssim \left(\frac{h}{c_h}\right)^{2(2-s)}.$$
Thus,
\begin{equation}
\label{eq:J31}
J_{3,1}\lesssim c_h^{N-2s}+ \left(\frac{h}{c_h}\right)^{2(2-s)}.
\end{equation}
In order to estimate $J_{3,2}$ first observe that $|y|\geq |x|-|x-y|>\frac{|x|}{2}\geq \frac{1}{3}$ and then $|x|<2$. We get:
\begin{equation*}
J_{3,2}\lesssim \int_{\frac{2}{3}\leq |x|\leq 2} \int_{\frac{1}{3}\leq |y|\leq 1,\ |x-y|<\frac{|x|}{2}} \frac{|D(I_h(\Psi_{\lambda_h,c_h,0})-\Psi_{\lambda_h,c_h,0})(\xi)|^2}{|x-y|^{N+2s-2}} \, \dd x\, \dd y
\end{equation*} 
where $\xi=(1-t)x+ty$ for some $t\in [0,1]$. We get that $|\xi|=|x+t(y-x)|\geq |x|-t|y-x|\geq \frac{|x|}{2}\geq \frac{1}{3}$ and $|\xi|\leq (1-t)|x|+t|y|\leq \max\{|x|,|y|\}\leq 2$, i.e. $\xi\in\{\frac{1}{3}\leq |x|\leq 2\}$.
Using the fact that $\int_{|z|<1} \frac{1}{|z|^{N+2s-2}}\, \dd z<\infty$, we obtain $$J_{3,2}\lesssim \|D(I_h(\Psi_{\lambda_h,c_h,0})-\Psi_{\lambda_h,c_h,0})\|^2_{L^{\infty}(\{\frac{1}{3}\leq |x|\leq 2\})}.$$ 
Since $D(I_h(\Psi_{\lambda_h,c_h,0}))$ vanishes outside $\BB_h$, we get that 
\begin{equation*}
\begin{split}
&\|D(I_h(\Psi_{\lambda_h,c_h,0})-\Psi_{\lambda_h,c_h,0})\|^2_{L^{\infty}(\{\frac{1}{3}\leq |x|\leq 2\})}\\
&\quad =\max\{\|D(I_h(\Psi_{\lambda_h,c_h,0})-\Psi_{\lambda_h,c_h,0})\|^2_{L^{\infty}(\{\frac{1}{3}\leq |x|\leq 1\}\cap \BB_h)}, \|D\Psi_{\lambda_h,c_h,0}\|^2_{L^{\infty}(\{\frac{1}{3}\leq |x|\leq 2\}\setminus \BB_h)}\}.
\end{split}
\end{equation*}
Using that $D^2\Psi_{\lambda_h,c_h,0}=D^2\Phi_{\lambda_h,c_h,0}$, we employ Lemma \ref{lem:interpolating-W2infty-functions} to obtain:
\begin{equation*}
\|D(I_h(\Psi_{\lambda_h,c_h,0})-\Psi_{\lambda_h,c_h,0})\|^2_{L^{\infty}(\{\frac{1}{3}\leq |x|\leq 1\}\cap \BB_h)}\lesssim h^2\|D^2\Phi_{\lambda_h,c_h,0}\|^2_{L^{\infty}(\frac{1}{3}-h\leq |x|\leq 1)}.
\end{equation*}
Using now estimate \eqref{second.derivative} we get, for $\frac{1}{3}-h\leq |x|\leq 1$
\begin{equation*}
|D^2\Phi_{\lambda_h,c_h,0}(x)|\lesssim \frac{|\lambda_h|}{c_h^2} \frac{c^{N+2-2s}_h}{(c^2_h+|x|^2)^{\frac{N+2-2s}{2}}}\lesssim c^{\frac{N-2s}{2}}_h.
\end{equation*}
Thus, we obtain
\begin{equation*}
\|D(I_h(\Psi_{\lambda_h,c_h,0})-\Psi_{\lambda_h,c_h,0})\|^2_{L^{\infty}(\{\frac{1}{3}\leq |x|\leq 1\}\cap \BB_h)} \lesssim h^{2} c_h^{N-2s}\lesssim c_h^{N-2s}.
\end{equation*}
Using that $D\Psi_{\lambda_h,c_h,0}=D\Phi_{\lambda_h,c_h,0}$ and \eqref{eq:B-Bh}
we get:
\begin{equation*}
\|D\Psi_{\lambda_h,c_h,0}\|^2_{L^{\infty}(\{\frac{1}{3}\leq |x|\leq 2\}\setminus \BB_h)}\leq\|D\Phi_{\lambda_h,c_h,0}\|^2_{L^{\infty}(\{1-\theta h^2\leq |x|\leq 2\})}.
\end{equation*}
For $1-\theta h^2\leq |x|\leq 2$ we get, using Lemma \ref{lem:GradHessMinimizers}:
\begin{equation*}
|D\Phi_{\lambda_h,c_h,0}(x)|\sim \frac{|\lambda_h|}{c^2_h} \frac{c^{N-2s+2}_h}{(c_h^2+|x|^2)^{\frac{N-2s+2}{2}}} |x|\sim c^{\frac{N-2s}{2}}_h.
\end{equation*}
Thus, we get that 
\begin{equation*}
\|D\Psi_{\lambda_h,c_h,0}\|^2_{L^{\infty}(\{\frac{1}{3}\leq |x|\leq 2\}\setminus \BB_h)}\lesssim c_h^{N-2s},
\end{equation*}
and then 
\begin{equation}
\label{eq:J32}
J_{3,2}\lesssim c_h^{N-2s}.
\end{equation}
The result now follows from \eqref{eq:J31} and \eqref{eq:J32}.
\end{proof}

Finally, we prove  $\|I_h(\Psi_{\lambda_h,c_h,0})\|_{L^{2_s^*}(\RR^N)}$ is like a constant.

\begin{proposition}\label{prop:estimate-from-above-denominator} 
The following holds:
   \[ \|I_h(\Psi_{\lambda_h,c_h,0})\|_{L^{2_s^*}(\RR^N)}\simeq 1.\]  
\end{proposition}
\begin{proof}
Note first that, since $I_h(\Psi_{\lambda_h,c_h,0})$ is extended with 0 outside $\BB_h$, we get that $\|I_h(\Psi_{\lambda_h,c_h,0})\|_{L^{2^*_s}(\RR^N)}=\|I_h(\Psi_{\lambda_h,c_h,0})\|_{L^{2^*_s}(\BB)}$. Next, using the triangle inequality we obtain: 
\begin{equation*}
\begin{split}
\left|1- \|I_h(\Psi_{\lambda_h,c_h,0})\|_{L^{2^*_s}(\RR^N)}\right|& =\left|\|\Psi_{\lambda_h,c_h,0}\|_{L^{2^*_s}(\BB)}-\|I_h(\Psi_{\lambda_h,c_h,0})\|_{L^{2^*_s}(\BB)}\right|\\
& \leq \|\Psi_{\lambda_h,c_h,0}-I_h(\Psi_{\lambda_h,c_h,0})\|_{L^{2^*_s}(\BB)}.
\end{split}
\end{equation*}
We bound the last norm as follows:
\begin{equation*}
\|\Psi_{\lambda_h,c_h,0}-I_h(\Psi_{\lambda_h,c_h,0})\|_{L^{2^*_s}(\BB)} \leq \|\Psi_{\lambda_h,c_h,0}-I_h(\Psi_{\lambda_h,c_h,0})\|_{L^{2^*_s}(\BB_h)}+\|\Psi_{\lambda_h,c_h,0}\|_{L^{2_s^*}(\BB\setminus\BB_h)}.
\end{equation*}
In order to estimate the first term in the above inequality we use Lemma \ref{lem:interpolatior-on-polyhedron-Lq-error} for $q=2_s^*$ and we obtain:
\begin{equation*}
\|\Psi_{\lambda_h,c_h,0}-I_h(\Psi_{\lambda_h,c_h,0})\| _{L^{2^*_s}(\BB_h)}\lesssim h^2c_h^{-2}.
\end{equation*}
In order to bound the second term we use \eqref{eq:B-Bh}
to obtain
\begin{equation*}
\begin{split}
\|\Psi_{\lambda_h,c_h,0}\|^{2_s^*}_{L^{2_s^*}(\BB\setminus\BB_h)}&\lesssim |\lambda_h|^{2_s^*} \int_{\{1-\theta h^2\leq |x|\leq 1\}} \left[\left(1+\frac{|x|^2}{c_h^2}\right)^{-\frac{N-2s}{2}}-\left(1+\frac{1}{c_h^2}\right)^{-\frac{N-2s}{2}}\right]^{2_s^*} \, \dd x\\ 
&\lesssim (c_h^{-\frac{N-2s}{2}})^{\frac{2N}{N-2s}} \int_{\{1-\theta h^2\leq |x|\leq 1\}} \left(1+\frac{|x|^2}{c_h^2}\right)^{-N} \, \dd x\lesssim c_h^{N}.
\end{split}
\end{equation*}
Since $h<<c_h \lesssim 1$ we obtain 
  $  \|I_h(\Psi_{\lambda_h,c_h,0})\|_{L^{2_s^*}(\RR^N)}\simeq  1$ 
  and the proof is finished.
\end{proof}

\section{Lower Bound in Theorem \ref{thm:main}}
\label{sec:estimatefrombelow}
Let $u_h\in V_h$ be a minimizer for $S_{N,s,h}$ with $\|u_h\|_{L^{2_s^*}(\BB_h)}=1$. Then we get
$$S_{N,s,h}=[u_h]^2_{\dot{H}^s(\RR^N)}.$$ Moreover, using the estimate in the previous section we obtain
\begin{equation}
\label{eq:lowerbound}
h^\alpha\gtrsim S_{N,s,h}-S_{N,s}=\delta(u_h)\gtrsim (d(u_h,\mathcal{M}))^2=\inf_{\lambda,c,X_0} \left[u_h-\Phi_{\lambda,c,X_0}\right]_{\dot{H}^s(\RR^N)}^2.
\end{equation} 

Let us denote by $\Phi_{\lambda_h,c_h,X_h}$ the function for which the minimum is achieved.
In the following we will show that 
\begin{equation}
\label{eq:claim-lower-bound}
[u_h-\Phi_{\lambda_{h},c_{h},X_{h}}]^2_{\dot{H}^s(\RR^N)}\gtrsim c^{N-2s}_h + \left(\frac{h}{c_h}\right)^{4-2s}\gtrsim h^\alpha.
\end{equation}
We will  divide the proof in four steps. 

\textbf{Step I}. We prove that $\|\Phi_{\lambda_{h},c_{h},X_{h}}\|_{L^{2_s^*}(\RR^N)}\simeq 1$ and that the sequence $(c_{h})_h$ is uniformly bounded as $h\to 0$.

Using the Fractional Sobolev Inequality (i.e. Proposition \ref{prop:Fractional-Sobolev-Inequality}) for the function $u_h-\Phi_{\lambda_{h},c_{h},X_{h}}$ we get that 
\begin{equation}
\label{eq:use-of-frac-sob-in-lower-bound}
h^{\frac{\alpha}{2}}\gtrsim [u_h-\Phi_{\lambda_{h},c_{h},X_{h}}]_{\dot{H}^s(\RR^N)}\gtrsim \| u_h-\Phi_{\lambda_{h},c_{h},X_{h}}\|_{L^{2_s^*}(\RR^N)}.
\end{equation}
Since $u_h$ is supported in $\BB_h$ we have
$$\| u_h-\Phi_{\lambda_{h},c_{h},X_{h}}\|^{2_s^*}_{L^{2_s^*}(\RR^N)}=\|\Phi_{\lambda_{h},c_{h},X_{h}}\|^{2_s^*}_{L^{2_s^*}(\BB_h^c)}+\|u_h-\Phi_{\lambda_{h},c_{h},X_{h}}\|^{2_s^*}_{L^{2_s^*}(\BB_h)}.$$
This implies
\begin{equation}
\label{eq:2s-norm-phi-on-complement}
\|\Phi_{\lambda_{h},c_{h},X_{h}}\|_{L^{2_s^*}(\BB_h^c)}\lesssim h^{\frac{\alpha }{2}}.
\end{equation} 
Since $\|u_h\|_{L^{2_s^*}(\BB_h)}=1$, we also get that $\|\Phi_{\lambda_{h},c_{h},X_{h}}\|_{L^{2_s^*}(\BB_h)}\simeq 1$ and then $S_{N,s}\simeq [\Phi_{\lambda_{h},c_{h},X_{h}}]_{\dot{H}^s(\RR^N)}^2$.
The  $L^{2_s^{*}}(\RR^N)-$norm of $\Phi_{\lambda_{h},c_{h},X_{h}}$ satisfies
\begin{equation}\label{eq:4}
1\simeq \|\Phi_{\lambda_{h},c_{h},X_{h}}\|^{2_s^{*}}_{L^{2_s^{*}}(\RR^N)}\sim |\lambda_{h}|^{\frac{2N}{N-2s}} c_{h}^N.
\end{equation}
Thus, $|\lambda_{h}|\sim c_{h}^{-\frac{N-2s}{2}}$.
Using that $\|\Phi_{\lambda_{h},c_{h},X_{h}}\|_{L^{2_s^*}(\BB_h)}\simeq 1$ we obtain 
\begin{equation*}
1\simeq \int_{\BB_h} |\lambda_{h}|^{\frac{2N}{N-2s}} \left(1+\frac{|x-X_{h}|^2}{c^2_{h}}\right)^{-N} \, \dd x\lesssim c^{-N}_{h} \int_{\BB_h} 1 \, \dd x\lesssim c^{-N}_{h}.
\end{equation*}
Thus, there exists $C_0>0$ such that $c_{h}\leq C_0$ for any $h<<1$.

\textbf{Step II}. We prove that $X_{h}\in \BB_h$ and $c_h\xrightarrow{h\to 0} 0$.

If $X_{h}\notin \BB_h$, then, since $\BB_h$ is convex, there exists a half-space, $\mathcal{H}\coloneqq\{x\in \RR^N: (x-X_{h})w\geq 0\}$ for some $w\in\RR^N,\ w\neq 0$, that passes through $X_{h}$ and which is contained in $\BB_h^c$. Using \eqref{eq:2s-norm-phi-on-complement} we get:
 \begin{align*}
h^{\frac{N\alpha}{N-2s}}&\gtrsim \int_{\BB_h^c} |\Phi_{\lambda_{h},c_{h},X_{h}}(x)|^{\frac{2N}{N-2s}} \, \dd x\\
&\gtrsim \int_{\mathcal{H}} |\lambda_{h}|^{\frac{2N}{N-2s}} \left(1+\frac{|x-X_{h}|^2}{c_{h}^2}\right)^{-N} \, \dd x\\
&\gtrsim c_{h}^{N} \int_{\mathcal{H}_0}  \frac{1}{\left(c_{h}^2+|x|^2\right)^N} \, \dd x\\
&\gtrsim \int_{\mathcal{H}_0 } \frac{1}{(1+|x|^2)^N}\, \dd x\gtrsim 1.
\end{align*}
where $\mathcal{H}_0\coloneqq \{x\in \RR^N: xw\geq 0\}$ is the half-space passing through the origin.
As a consequence $X_{h}\in \BB_h$.

Let us suppose that, up to a subsequence, denoted the same, there exists $0<C_2\leq C_0$ such that $c_h\geq C_2$. Denoting $M_0\coloneqq \max\left\{2,\sqrt{C_0 C_2}\right\}$ and using the fact that $X_h\in \BB_h$, we get that $B^c_{M_0}(X_h)\subset \BB^c_h$. Using that $|\lambda_h|\sim c^{-\frac{N-2s}{2}}_h$, we obtain a contradiction with \eqref{eq:2s-norm-phi-on-complement} as follows:
\begin{equation*}
\begin{aligned}
\int_{\BB_h^c} |\Phi_{\lambda_h,c_h,X_h}(x)|^{\frac{2N}{N-2s}} \, \dd x&=\int_{\BB_h^c} |\lambda_h|^{\frac{2N}{N-2s}} \left(1+\frac{|x-X_h|^2}{c_h^2}\right)^{-N} \, \dd x\\
&\sim \int_{\BB_h^c} \left(c_h+\frac{|x-X_h|^2}{c_h}\right)^{-N} \, \dd x\\
&\geq \int_{\BB_h^c}\left(C_0+\frac{|x-X_h|^2}{C_2}\right)^{-N} \, \dd x\\
&\geq \int_{B^c_{M_0}(X_h)} C_2^N \left(M_0^2+|x-X_h|^2\right)^{-N} \, \dd x\\
&\gtrsim \int_{B^c_{M_0}(X_h)} |x-X_h|^{-2N} \, \dd x\\
&\gtrsim 1.
\end{aligned}
\end{equation*}
Thus, $c_h\xrightarrow{h\to 0} 0$.

Before entering into the last two steps of the proof of the lower bound in Theorem \ref{thm:main}, we recall here Lemma 2.2 from \cite{ignat2025optimalconvergenceratesfinite}:
\begin{lemma}
\label{lem:bound-on-gradient}
    For any real number $1<p<\infty$, there exists a positive constant $C(p)$ such that for any simplex $T\in\TT_h$ and any function $u\in C^2(T)$, the following estimate holds:
    \begin{equation*}
    \min_{A\in \RR^N} \int_{T} |Du(x)-A|^p \, \dd x\geq C(p)\rho^{N+p}_T\max_{\xi\in \mathbb{S}^{N-1}} \min_{x\in T} |\xi^T D^2u(x)\xi|^p
    \end{equation*}
\end{lemma}
\begin{remark} In \cite{ignat2025optimalconvergenceratesfinite}, the one-dimensional case is proved in Lemma A.1.
\end{remark}

{\begin{remark}
\label{rem:lemma41-balls}
The proof of Lemma \ref{lem:bound-on-gradient} in \cite{ignat2025optimalconvergenceratesfinite} only uses a ball contained in $T$. For a fixed direction $\xi$, along each segment parallel to $\xi$ the function $\xi\cdot(Du-A)$ is monotone, with derivative bounded from below by $\min|\xi^TD^2u\,\xi|$. One then integrates over the segments crossing the ball. Consequently, Lemma \ref{lem:bound-on-gradient} also holds when $T$ is replaced by any ball $B_r(y_0)\subset\RR^N$ and $\rho_T$ is replaced by $2r$: for every $u\in C^2(\overline{B_r(y_0)})$,
\[
\min_{A\in \RR^N} \int_{B_r(y_0)} |Du(x)-A|^p \, \dd x\geq C(p)\,(2r)^{N+p}\max_{\xi\in \mathbb{S}^{N-1}} \min_{x\in \overline{B_r(y_0)}} |\xi^T D^2u(x)\xi|^p,
\]
where the constant $C(p)$ may also depend on $N$. In Steps III and IV below we use Lemma \ref{lem:bound-on-gradient} in this form. The direction $\xi$ in the right-hand side must be the same for all the points of the ball. Such a direction is provided by Lemma \ref{lem:finitecovering} when $N=1$ and by Lemma \ref{lem:orthogonal-direction} when $N\geq 2$.
\end{remark}
}

\textbf{Step III}. We prove that the sequence $(\frac{h}{c_{h}})_h$ is uniformly bounded. 

Let us assume that, up to a subsequence (denoted the same), $\lim_{h\to 0} \frac{h}{c_{h}}=\infty$.
By \eqref{eq:use-of-frac-sob-in-lower-bound} we get that 
\begin{equation}
\label{eq:step-3}
h^{\frac{\alpha}{2}}\gtrsim \|u_h-\Phi_{\lambda_{h},c_{h},X_{h}}\|_{L^{2_s^*}(\BB_h)}.
\end{equation}

Since $X_{h}\in\BB_h$, there exists a simplex $T_{0,h}\in \mathcal{T}_h$ such that $X_{h}\in T_{0,h}$.\\ 
\textit{Case a.} $N=1$\\
Let $T_{0,h}=[a_h,b_h]$. By the quasi-uniformity of the mesh, $|T_{0,h}|=b_h-a_h\geq \rho h$. Since $X_h\in T_{0,h}$, at least one of the numbers $b_h-X_h$ and $X_h-a_h$ is greater than or equal to $\frac{|T_{0,h}|}{2}\geq \frac{\rho h}{2}$. Since $\frac{h}{c_h}\to\infty$, we have $c_h<\frac{\rho h}{2}$ for $h$ small enough. We define
\begin{equation*}
Q_{0,h}\coloneqq
\begin{cases}
(X_h,X_h+c_h), &\ \text{if } b_h-X_h\geq \frac{|T_{0,h}|}{2},\\
(X_h-c_h,X_h), &\  \text{otherwise},
\end{cases}
\qquad
X_{1,h}\coloneqq
\begin{cases}
X_h+\frac{c_h}{4}, &\ \text{if } b_h-X_h\geq \frac{|T_{0,h}|}{2},\\
X_h-\frac{c_h}{4}, &\  \text{otherwise}.
\end{cases}
\end{equation*}
Then $Q_{0,h}\subset T_{0,h}$ and $|Q_{0,h}|=c_h$. Note that we cannot take the interval $(X_h-c_h,X_h+c_h)$, since it is not contained in $T_{0,h}$ when the distance from $X_h$ to an end-point of $T_{0,h}$ is smaller than $c_h$. Moreover,
\begin{equation}
\label{eq:one-dim-rel-1}
Q_{0,h}\subset B_{2c_h}(X_h)
\end{equation}
and
\begin{equation}
\label{eq:one-dim-rel-2}
Q_{0,h}\supset B_{\frac{c_h}{4}}(X_{1,h}),\qquad |x-X_h|\leq |x-X_{1,h}|+|X_{1,h}-X_h|<\frac{c_h}{2}\quad \text{for all } x\in B_{\frac{c_h}{4}}(X_{1,h}).
\end{equation}

\textit{Case b.} $N\geq 2$\\
For $X_0\in \RR^N, R>0, \alpha\in (0,1)$ and $\xi\in \mathbb{S}^{N-1}$, we define the cone $$C_R(\alpha,\xi,X_0)\coloneqq \{x\in \RR^N: |x-X_0|<R,(x-X_0)\xi \geq \alpha |x-X_0|\}.$$ For the reference simplex $\widehat{T}$, there exists $R>0$ and $\alpha\in (0,1)$ such that $$\forall\ \hat{x}\in \widehat{T},\ \exists\ \hat{\xi}=\hat{\xi}(\hat{x})\in \mathbb{S}^{N-1}\ \text{such that } C_R(\alpha,\hat{\xi}, \hat{x})\subset \widehat{T}.$$
We transfer this property to any $T\in \TT_h$ by using the properties of the matrices $B_T$ in \eqref{eq:norm-BT} and \eqref{eq:det-BT}. Indeed, $F_T$ maps $C_R(\alpha,\hat\xi,\hat x)$ onto a set which contains a cone with vertex $F_T(\hat x)$, whose height and aperture are bounded from below in terms of $R$, $\alpha$, $\|B_T^{-1}\|^{-1}\geq \rho_T/h_{\widehat T}$ and $\|B_T\|\,\|B_T^{-1}\|\leq \sigma h_{\widehat T}/\rho_{\widehat T}$. In view of these properties, there exists $R_\sigma>0$ and $\alpha_{\sigma}\in (0,1)$ such that $$\forall\ T\in \TT_h\ \text{and } \forall\  x\in T,\ \exists\ \xi=\xi(x)\in \mathbb{S}^{N-1}\ \text{such that } C_{R_{\sigma} h}(\alpha_{\sigma}, \xi,x)\subset T.$$ Here the height of the cones is proportional to $h_T\geq \rho h$, so that $R_\sigma$ depends on $\sigma$ and $\rho$.
Since $h>>c_h$, we get that $$C_{c_h}(\alpha_\sigma, \xi(X_h),X_h)\subset C_{R_\sigma h}(\alpha_\sigma, \xi(X_h),X_h)\subset T_{0,h}.$$
In the following we will construct a cube $Q_{0,h}\subset C_{c_h}(\alpha_\sigma, \xi(X_h),X_h)$ of side-length of order $\mathcal{O}(c_h)$. Let the center of the cube be $$X_{1,h}\coloneqq X_h+c_1 c_h\xi(X_h)$$ and the side-length be $$l_h\coloneqq c_2 c_h,$$ where $c_1$ and $c_2$ are some positive constants which depend only on $\alpha_\sigma$ and $N$ and which will be chosen later on. Thus, for any point $x\in Q_{0,h}$, we get that $|x-X_{1,h}|\leq \frac{l_h\sqrt{N}}{2}$ and then 
\begin{equation}
\label{eq:distance-on-Q0h}
|x-X_h|\leq |x-X_{1,h}|+|X_{1,h}-X_h|\leq \frac{l_h\sqrt{N}}{2}+c_1 c_h.
\end{equation}
Imposing the condition that any point $x\in Q_{0,h}$ should satisfy $|x-X_h|<c_h$, we get a first condition on $c_1$ and $c_2$: $$\frac{c_2\sqrt{N}}{2}+c_1<1.$$ Next, for any point $x\in Q_{0,h}$, we have that $$(x-X_h) \xi(X_h)=(x-X_{1,h})\xi(X_h)+(X_{1,h}-X_h)\xi(X_h)\geq c_1c_h-\frac{l_h\sqrt{N}}{2}.$$ Taking into account that $(x-X_h)\xi(X_h)\geq \alpha_\sigma |x-X_h|$ for any $x\in Q_{0,h}$ and using the fact that $|x-X_h|\leq \frac{l_h\sqrt{N}}{2}+c_1c_h$, we get a second restriction on $c_1$ and $c_2$: $$c_1\geq \frac{c_2\sqrt{N} (\alpha_\sigma+1)}{2(1-\alpha_\sigma)}.$$
Choosing now $$c_2=\frac{1-\alpha_\sigma}{2\sqrt{N}}$$ and $$c_1=\frac{\alpha_\sigma+1}{2},$$ we get the desired conclusion. 
Moreover, using \eqref{eq:distance-on-Q0h}, we get that 
\begin{equation}
\label{eq:N-dim-rel-1}
Q_{0,h}\subset B_{2 c_h}(X_h).
\end{equation}
Also,
\begin{equation}
\label{eq:N-dim-rel-2}
Q_{0,h}\supset B_{\frac{l_h}{2\sqrt{N}}}(X_{1,h}).
\end{equation}
In either of the two cases (\textit{a} and \textit{b}), from \eqref{eq:one-dim-rel-1} and \eqref{eq:N-dim-rel-1}, we get that
\begin{equation}
\label{eq:general-rel-1}
Q_{0,h}\subset B_{2c_h}(X_h).
\end{equation}

We use Lemma \ref{lem:interpolationoncubes} for $\Omega=Q_{0,h}$. Taking $p=2_s^{*}$ and $u\coloneqq u_h-\Phi_{\lambda_{h},c_{h},X_{h}}$ in the above-mentioned Lemma and using the fact that $u_h$ is linear on $T_{0,h}$ we get 
\begin{equation}
\label{eq:24}
\begin{split}
     \|D(u_h-\Phi_{\lambda_{h},c_h,X_{h}})\|_{L^{2_s^*}(Q_{0,h})} & \leq  \tilde{C} c_{h}^{-1} \|u_h-\Phi_{\lambda_{h},c_{h},X_{h}}\|_{L^{2_s^*}(Q_{0, h})}\\
     & +\tilde{C} \|u_h-\Phi_{\lambda_{h},c_{h},X_{h}}\|^{\frac{1}{2}}_{L^{2_s^*}(Q_{0,h})} \|D^2 \Phi_{\lambda_{h},c_{h},X_{h}}\|^{\frac{1}{2}}_{L^{2_s^*}(Q_{0,h})}
     \end{split}
\end{equation}
where $\tilde{C}>0$ is a constant which depends only on $N,s,\sigma,\rho$, but not on $h$.

Using \eqref{eq:general-rel-1}, estimate \eqref{second.derivative} and the fact that $|\lambda_{h}|\sim c^{-\frac{N-2s}{2}}_{h}$ we get that 
\begin{equation*}
\begin{split}
\|D^2\Phi_{\lambda_{h},c_{h},X_{h}}\|^{2_s^*}_{L^{2_s^*}(Q_{0,h})}&\lesssim \Big(\frac{|\lambda_h|}{c^2_h}\Big)^{2_s^*} \int_{Q_{0,h}} \left(1+\frac{|x-X_{h}|^2}{c_{h}^2}\right)^{-\frac{N(N-2s+2)}{N-2s}} \, \dd x\\
&\lesssim \Big(\frac{|\lambda_h|}{c^2_h}\Big)^{2_s^*} \int_{B_{2c_h}(X_h)} \left(1+\frac{|x-X_{h}|^2}{c_{h}^2}\right)^{-\frac{N(N-2s+2)}{N-2s}} \, \dd x\\
& \lesssim \Big(\frac{|\lambda_h|}{c^2_h}\Big)^{2_s^*} c_h^{\frac{2N(N-2s+2)}{N-2s}} \int_{0}^{2c_{h}} \frac{\sigma^{N-1}}{(c^2_{h}+\sigma^2)^{\frac{N(N-2s+2)}{N-2s}}} \, \dd \sigma\lesssim c_h^{-\frac{4N}{N-2s}}.
\end{split}
\end{equation*}
Thus,
\begin{equation}
\label{eq:28}
\|D^2\Phi_{\lambda_{h},c_{h},X_{h}}\|_{L^{2_s^*}(Q_{0,h})}\leq \tilde{C}_1 c_{h}^{-2}
\end{equation}
where $\tilde{C}_1$ is a positive constant which does not depend on $h$.

If $N=1$, we let $X_{1,h}$ be the point defined in Case a and we set $c_2\coloneqq \frac{1}{2}$, so that $\frac{c_2c_h}{2\sqrt{N}}=\frac{c_h}{4}$. With this notation, we use \eqref{eq:one-dim-rel-2}, \eqref{eq:N-dim-rel-2} (recall that $l_h=c_2c_h$ when $N\geq 2$) to get in each of the two cases (\textit{a} and \textit{b}):
\begin{equation}
\label{eq:general-rel-2}
Q_{0,h}\supset B_{\frac{c_2 c_h}{2\sqrt{N}}}(X_{1,h}).
\end{equation}

Using \eqref{eq:general-rel-2} and the fact that $Du_h$ is constant on $T_{0,h}$, we can apply Lemma \ref{lem:bound-on-gradient} on the ball $B_{\frac{c_2 c_h}{2\sqrt{N}}}(X_{1,h})$ (see Remark \ref{rem:lemma41-balls}), together with Lemma \ref{lem:finitecovering} if $N=1$ and Lemma \ref{lem:orthogonal-direction} if $N\geq 2$. When $N\geq 2$, we apply Lemma \ref{lem:orthogonal-direction} with $X_0=X_h$, $y_0=X_{1,h}$ and $r=\frac{c_2c_h}{2\sqrt{N}}=\frac{(1-\alpha_\sigma)c_h}{4N}$. We take any direction $\xi\in\mathbb{S}^{N-1}$ with $\xi\perp\xi(X_h)$; recall that $X_{1,h}-X_h=c_1c_h\,\xi(X_h)$. Condition \eqref{eq:cond-orth} holds, since
\[
2(N+2)r^2\leq \frac{2(N+2)}{16N^2}\,c_h^2\leq c_h^2.
\]
When $N=1$, we can apply Lemma \ref{lem:finitecovering} because, for $x\in B_{\frac{c_h}{4}}(X_{1,h})$, we have $|x-X_h|<\frac{c_h}{2}$ by \eqref{eq:one-dim-rel-2}. In this case, Lemma \ref{lem:bound-on-gradient} is applied on the interval $B_{\frac{c_h}{4}}(X_{1,h})\subset T_{0,h}$, where it reduces to the one-dimensional estimate in \cite[Lemma A.1]{ignat2025optimalconvergenceratesfinite}.
We get:
\begin{equation*}
\begin{split}
\|D(u_h-\Phi_{\lambda_{h},c_{h},X_{h}})\|^{2_s^*}_{L^{2_s^*}(Q_{0,h})} &\geq \|D(u_h-\Phi_{\lambda_{h},c_{h},X_{h}})\|^{2_s^*}_{L^{2_s^*}\left(B_{\frac{c_2 c_h}{2\sqrt{N}}}(X_{1,h})\right)}\\
    &\geq \inf_{A_0\in \RR^N} \int_{B_{\frac{c_2 c_h}{2\sqrt{N}}}(X_{1,h})} |A_0-D\Phi_{\lambda_{h},c_{h},X_{h}} (x)|^{2_s^*} \, \dd x\\
    & \gtrsim c_{h}^{N+\frac{2N}{N-2s}} \Big(\frac{|\lambda_h|}{c^2_h}\Big)^{2_s^*} \min_{x\in \overline{B}_{\frac{c_2 c_h}{2\sqrt{N}}}(X_{1,h})}\left\{\left(1+\frac{|x-X_{h}|^2}{c_{h}^2}\right)^{-\frac{N(N-2s+2)}{N-2s}}\right\}\\
    & \gtrsim c^{-\frac{2N}{N-2s}}_{h},
    \end{split}
\end{equation*}
where we have used that $$\min_{x\in \overline{B}_{\frac{c_2 c_h}{2\sqrt{N}}}(X_{1,h})}\left\{\left(1+\frac{|x-X_{h}|^2}{c_{h}^2}\right)^{-\frac{N(N-2s+2)}{N-2s}}\right\}\gtrsim 1.$$

Thus, there exists a positive constant $\tilde{C}_2$ which does not depend on $h$ such that
\begin{equation}
\label{eq:30}
\|D(u_h-\Phi_{\lambda_{h},c_{h},X_{h}})\|_{L^{2_s^*}(Q_{0,h})}\geq \tilde{C}_2 c^{-1}_{h}.
\end{equation}

From estimates \eqref{eq:24}, \eqref{eq:28} and \eqref{eq:30} we get the following quadratic inequality
\begin{equation}
   \tilde{C} b^2+\tilde{C}\tilde{C}_1^{1/2} b-\tilde{C}_2 \geq 0
\end{equation}
where $b\coloneqq \|u_h-\Phi_{\lambda_{h},c_{h},X_{h}}\|^{\frac{1}{2}}_{L^{2_s^*}(Q_{0,h})}\lesssim h^{\frac{\alpha}{4}}$.
Then
\begin{equation}
h^{\frac{\alpha}{4}}\gtrsim \tilde{C}b^2+\tilde{C}\tilde{C}_1^{1/2}b\geq \tilde{C}_2,
\end{equation}
which is false. Thus, there exists a positive constant $C_1$ such that $\frac{h}{c_{h}}\leq C_1$ for any $h\in (0,1)$ sufficiently small.

\textbf{Step IV}. We will prove the following estimates
\begin{equation}
\label{eq:26} 
I_{1h}\coloneqq\int_{\RR^N\setminus\BB_h}\int_{\RR^N\setminus\BB_h} \frac{|\Phi_{\lambda_{h},c_{h},X_{h}}(x)-\Phi_{\lambda_{h},c_{h},X_{h}}(y)|^2}{|x-y|^{N+2s}} \, \dd x \, \dd y\gtrsim c^{N-2s}_{h}
\end{equation}
and
 \begin{equation}
 \label{eq:17}
I_{2h}\coloneqq \int_{\BB_h}\int_{\BB_h} \frac{|(u_h-\Phi_{\lambda_{h},c_{h},X_{h}})(x)-(u_h-\Phi_{\lambda_{h},c_{h},X_{h}})(y)|^2}{|x-y|^{N+2s}} \, \dd x\, \dd y\gtrsim \left(\frac{h}{c_{h}}\right)^{4-2s}. 
\end{equation}
These give us the desired lower bound since
\[
[u_h-\Phi_{\lambda_h,c_h,X_h}]^2_{\dot{H}^s(\RR^N)}\geq s(1-s)(I_{1h}+I_{2h})\gtrsim c^{N-2s}_{h}+\left(\frac{h}{c_{h}}\right)^{4-2s}\gtrsim h^\alpha. 
\]

In the following we will prove the above inequalities.

Let $C_1$ be the constant given by Step III, so that $\frac{h}{c_h}\leq C_1$ for $h$ small enough, and set
\[
M\coloneqq \max\{C_1,1\}.
\]
Since $c_h\to 0$ as $h\to 0$ (Step II), choosing, if needed, a smaller value of $h$, we can assume that
\begin{equation}
\label{eq:choice-of-h-step4}
\frac{h}{c_h}\leq C_1\leq M\qquad\text{and}\qquad 16Mc_h\leq 1.
\end{equation}
Note that $M$ depends only on $N$, $s$, $\sigma$ and $\rho$, and not on $h$.

\textbf{Case 1. Proof of \eqref{eq:26}.}

To bound the integral on $(\RR^N\setminus \BB_h)\times (\RR^N\setminus \BB_h)$ we use the fact that the sequence $(c_{h})_h$ is uniformly bounded by $C_0$ and the fact that $X_{h}\in \BB_h$, so that $\RR^N\setminus B_{2}(X_{h})\subseteq \RR^N\setminus\BB_h$. We obtain:
\begin{equation*}
\begin{split}
& \int_{\RR^N\setminus \BB_h}\int_{\RR^N\setminus\BB_h} \frac{|\Phi_{\lambda_{h},c_{h},X_{h}}(x)-\Phi_{\lambda_{h},c_{h},X_{h}}(y)|^2}{|x-y|^{N+2s}} \, \dd x \, \dd y\\
 & \gtrsim c_{h}^{N-2s}\int_{\RR^N\setminus B_{2}(X_{h})} \int_{\RR^N\setminus B_{2}(X_{h})}\left|\frac{1}{\left(c^2_{h}+|x-X_{h}|^2\right)^{\frac{N-2s}{2}}}-\frac{1}{\left(c^2_{h}+|y-X_{h}|^2\right)^{\frac{N-2s}{2}}}\right|^2 \frac{\, \dd x\, \dd y}{|x-y|^{N+2s}}\\
 &\gtrsim c^{N-2s}_{h} \int_{\RR^N\setminus B_{2}(0)} \int_{\RR^N\setminus B_{2}(0)} \left|\frac{1}{\left(c^2_{h}+|x|^2\right)^{\frac{N-2s}{2}}}-\frac{1}{\left(c^2_{h}+|y|^2\right)^{\frac{N-2s}{2}}}\right|^2\frac{ \, \dd x\, \dd y}{|x-y|^{N+2s}}\\
 & \gtrsim c^{N-2s}_{h} \int_{\left(\RR^N\setminus B_{2}(0)\right)^2} \left|\frac{1}{\left(c^2_{h}+|x|^2\right)^{\frac{N-2s}{2}}}-\frac{1}{\left(c^2_{h}+|y|^2\right)^{\frac{N-2s}{2}}}\right|^2\frac{ \, \dd x\, \dd y}{(|x|+|y|)^{N+2s}}\\ 
 & \gtrsim  c^{N-2s}_{h} \int_{\max\{2,C_0\}}^{\infty} \int_{\max\{2,C_0\}}^{\infty} \frac{\sigma^{N-1} t^{N-1}}{(\sigma+t)^{N+2s}} \left|\frac{1}{(c^2_{h}+\sigma^2)^{\frac{N-2s}{2}}}-\frac{1}{(c^2_{h}+t^2)^{\frac{N-2s}{2}}}\right|^2 \, \dd t \, \dd \sigma.
\end{split}
\end{equation*}

Let $g:[\max\{2,C_0\},\infty)\rightarrow\RR,\ g(p)=\left(c_{h}^2+p^2\right)^{-\frac{N-2s}{2}}$. Then $g'(p)=-(N-2s)\left(c_{h}^2+p^2\right)^{-\frac{N-2s+2}{2}} p$ and $$
\begin{aligned}
|g(\sigma)-g(t)|&=\left|(N-2s)\int_{\sigma}^{t} p\left(c_{h}^2+p^2\right)^{-\frac{N-2s+2}{2}} \, \dd p\right|\gtrsim \left|\int_{\sigma}^{t} \frac{p}{p^{N-2s+2}} \,\dd p\right|\\
&=\frac{1}{N-2s}\left|\frac{1}{\sigma^{N-2s}}-\frac{1}{t^{N-2s}}\right|\geq 0.
\end{aligned}
$$

Thus, we obtain:
\begin{equation}
\begin{split}
I_{1h}& \gtrsim c^{N-2s}_{h} \int_{\max\{2,C_0\}}^{\infty} \int_{\max\{2,C_0\}}^{\infty} \frac{\sigma^{N-1}t^{N-1}}{(\sigma+t)^{N+2s}} \left|\frac{1}{\sigma^{N-2s}}-\frac{1}{t^{N-2s}}\right|^2 \, \dd \sigma \, \dd t\gtrsim c^{N-2s}_{h}.
\end{split}
\end{equation}

\textbf{Case 2. Proof of \eqref{eq:17}.} Observe that
\[
I_{2h}\geq  \sum_{T\in \mathcal{T}_h} \int_{T}\int_{T} \frac{|(u_h-\Phi_{\lambda_{h},c_{h},X_{h}})(x)-(u_h-\Phi_{\lambda_{h},c_{h},X_{h}})(y)|^2}{|x-y|^{N+2s}} \, \dd x \, \dd y.
\]
Let $M=\max\{C_1,1\}$ be the constant in \eqref{eq:choice-of-h-step4}. We will neglect the integrals on $T\times T$ for which the simplex $T$ has an empty intersection with the set $\{x\in \RR^N: |x-X_{h}|\geq 2M c_h\}$. The main tools that we will use to bound these integrals are Lemma \ref{lem:interpolationoncubes} and Poincar\'{e}'s Inequality for fractional Sobolev spaces in Theorem \ref{thm:Poincareinequality}.

Using Poincar\'{e}'s Inequality  \ref{thm:Poincareinequality} with $p=2$, $E=\Omega=T\in\mathcal{T}_h$ and $u=u_h-\Phi_{\lambda_{h},c_{h},X_{h}}$, we get that 
\begin{equation}
\label{eq:25}
\begin{split}
 & \int_{T}\int_{T} \frac{|(u_h-\Phi_{\lambda_{h},c_{h},X_{h}})(x)-(u_h-\Phi_{\lambda_{h},c_{h},X_{h}})(y)|^2}{|x-y|^{N+2s}} \, \dd x \, \dd y\\
& \geq \frac{|T|}{h_T^{N+2s}} \int_{T} |(u_h-\Phi_{\lambda_{h},c_{h},X_{h}})(x)-(u_h-\Phi_{\lambda_{h},c_{h},X_{h}})_{T}|^2 \, \dd x\\
&\gtrsim h^{-2s} \int_{Q_T} |(u_h-\Phi_{\lambda_{h},c_{h},X_{h}})(x)-(u_h-\Phi_{\lambda_{h},c_{h},X_{h}})_{T}|^2 \, \dd x,
\end{split}
\end{equation}
where $(u_h-\Phi_{\lambda_{h},c_{h},X_{h}})_{T}\coloneqq A_T\in \RR$ is the mean of $u_h-\Phi_{\lambda_{h},c_{h},X_{h}}$ on $T$ and $Q_T$ is the maximal inscribed cube in the inscribed sphere of $T$. Note that, by the regularity and the quasi-uniformity assumptions on the mesh, the side-length of $Q_T$ is of order $\mathcal{O}(h)$.

We denote by $\mathcal{T}_{h,1}$ the set of the simplices $T\in \mathcal{T}_h$ which have a nonempty intersection with the set $\{x\in \RR^N: |x-X_{h}|\geq 2M c_h\}$. Any simplex $T\in\mathcal{T}_{h,1}$ is included in the exterior of the ball $B_{2M c_h-h}(X_{h})$ and for any $x\in T\in \mathcal{T}_{h,1}$ we have that:
\begin{equation*}
\max_{\overline{x}\in T} |\overline{x}-X_{h}|\leq |x-X_{h}|+h_T\leq |x-X_{h}|+h\leq 2|x-X_{h}|,
\end{equation*}
since $2M c_h\geq 2C_1 c_h\geq 2h$, and
\begin{equation*}
    |x-X_{h}|\leq \max_{\overline{x}\in T} |\overline{x}-X_{h}|\leq 2\min_{\underline{x}\in T} |\underline{x}-X_{h}|\leq 2|x-X_{h}|.
\end{equation*}
Note that the maximal ball inscribed in the cube $Q_T$, denoted by $\BB_Q$, has radius of order $\mathcal{O}(h)$. {More precisely, $\BB_Q=B_{r_T}(y_T)$ for some $y_T\in T$, with $r_T=\frac{\rho_T}{2\sqrt{N}}$, so that $\frac{h_T}{2\sqrt{N}\sigma}\leq r_T\leq \frac{h}{2\sqrt{N}}$. We apply Lemma \ref{lem:bound-on-gradient} on the ball $\BB_Q$ (see Remark \ref{rem:lemma41-balls}), together with Lemma \ref{lem:finitecovering} if $N=1$ and Lemma \ref{lem:orthogonal-direction} if $N\geq 2$. When $N\geq 2$, we apply Lemma \ref{lem:orthogonal-direction} with $X_0=X_h$, $B_r(y_0)=\BB_Q$ and a direction $\xi_T\in\mathbb{S}^{N-1}$ with $\xi_T\perp y_T-X_h$; the direction may change from one simplex to another. Condition \eqref{eq:cond-orth} holds: since $T$ lies outside $B_{2Mc_h-h}(X_h)$ and $Mc_h\geq C_1c_h\geq h$, we have
\[
\operatorname{dist}(X_h,\BB_Q)\geq 2Mc_h-h\geq Mc_h\geq h,
\]
and hence, as $N\geq 2$,
\[
2(N+2)r_T^2\leq \frac{N+2}{2N}\,h^2\leq h^2\leq \operatorname{dist}(X_h,\BB_Q)^2.
\]
When $N=1$, we can apply Lemma \ref{lem:finitecovering} because, for any $x\in T\in \TT_{h,1}$, we have that $$|x-X_h|\geq 2Mc_h-h=M c_h+(M c_h-h)\geq c_h.$$
Using these lemmas and $|\lambda_h|\sim c_h^{-\frac{N-2s}{2}}$,} we obtain for any $T\in\mathcal{T}_{h,1}$:
\begin{equation}
\begin{split}
  \int_{Q_T} |D(u_h-\Phi_{\lambda_{h},c_{h},X_{h}})(x)|^2 \, \dd x &\gtrsim\int_{\BB_Q} |D(u_h-\Phi_{\lambda_{h},c_{h},X_{h}})(x)|^2 \, \dd x \\
 &\gtrsim  h^{N+2} c_{h}^{-(N-2s+4)} \min_{x\in \overline{\BB}_Q} \left(1+\frac{|x-X_{h}|^2}{c_{h}^2}\right)^{-(N-2s+2)}\\
&\gtrsim h^{N+2} c_{h}^{-(N-2s+4)} \min_{x\in T} \left(1+\frac{|x-X_{h}|^2}{c_{h}^2}\right)^{-(N-2s+2)}\\
 & \gtrsim h^2 c_{h}^{-(N-2s+4)} \int_{T} \left(1+\frac{|x-X_{h}|^2}{c_{h}^2}\right)^{-(N-2s+2)} \, \dd x.
  \end{split}
\end{equation}
Denoting $f(x)=\left(1+\frac{|x-X_{h}|^2}{c^2_{h}}\right)^{-\frac{N-2s+2}{2}}$ the above inequality can be re-written as
\begin{equation}
\label{eq:34}
\|D(u_h-\Phi_{\lambda_{h},c_{h},X_{h}})\|_{L^2(Q_T)}\geq \tilde{C}_{3} h c^{-\frac{N-2s+4}{2}}_{h} \|f\|_{L^2(T)}
\end{equation}
where $\tilde{C}_3$ is a positive constant which does not depend neither on $h$ and neither on $T$.

Using estimate \eqref{second.derivative} we get that there exists a positive constant $\tilde{C}_4$ which does not depend neither on $h$ and neither on $T$, such that, for any $T\in\mathcal{T}_{h,1}$, we have: 
\begin{equation}
\label{eq:35}
    \|D^2\Phi_{\lambda_{h},c_{h},X_{h}}\|^{\frac{1}{2}}_{L^2(Q_T)}\leq \tilde{C}_4 c_{h}^{-\frac{N-2s+4}{4}} \|f\|^{\frac{1}{2}}_{L^2(T)}.
\end{equation}

We use Lemma \ref{lem:interpolationoncubes} with $\Omega=Q_T$, $p=2$ and $u\coloneqq u_h-\Phi_{\lambda_{h},c_{h},X_{h}}-A_T$. Estimates \eqref{eq:34} and \eqref{eq:35} show that, for any $T\in \TT_{h,1}$:
\begin{equation*}
\begin{split}
\tilde{C}h^{-1} \|u_h-\Phi_{\lambda_{h},c_{h},X_{h}}-A_T\|_{L^2(Q_T)} & + \tilde{C} \|u_h-\Phi_{\lambda_{h},c_{h},X_{h}}-A_T\|^{\frac{1}{2}}_{L^2(Q_T)} \tilde{C}_4 c_{h}^{-\frac{N-2s+4}{4}} \|f\|^{\frac{1}{2}}_{L^2(T)}\\
& - \tilde{C}_{3} h c^{-\frac{N-2s+4}{2}}_{h} \|f\|_{L^2(T)}\geq 0
\end{split}
\end{equation*}
 In order to solve the above quadratic inequality, we denote $\tilde{b}\coloneqq \|u_h-\Phi_{\lambda_{h},c_{h},X_{h}}-A_T\|^{\frac{1}{2}}_{L^2(Q_T)}> 0$. Thus, the inequality can be re-written as:
\begin{equation*}
\tilde{C}h^{-1}\tilde{b}^2+\tilde{C}\tilde{C}_4 c_{h}^{-\frac{N-2s+4}{4}}   \|f\|^{\frac{1}{2}}_{L^2(T)} \tilde{b}- \tilde{C}_{3} h c^{-\frac{N-2s+4}{2}}_{h} \|f\|_{L^2(T)}\geq 0.
\end{equation*}
The discriminant is 
\begin{equation*}
\Delta=c_{h}^{-\frac{N-2s+4}{2}} \|f\|_{L^2(T)} (\tilde{C}\tilde{C_4})^2+c_{h}^{-\frac{N-2s+4}{2}} \|f\|_{L^2(T)} 4 \tilde{C}\tilde{C}_3=\tilde{C}_6 c_{h}^{-\frac{N-2s+4}{2}} \|f\|_{L^2(T)},
\end{equation*} 
where $\tilde{C}_6=(\tilde{C}\tilde{C_4})^2+4 \tilde{C}\tilde{C}_3>(\tilde{C}\tilde{C}_4)^2>0$.
Since $\tilde{b}> 0$, we get that \[\tilde{b}\geq h c_{h}^{-\frac{N-2s+4}{4}}   \|f\|^{\frac{1}{2}}_{L^2(T)} \frac{\sqrt{\tilde{C}_6}-\tilde{C}\tilde{C_4}}{2\tilde{C}}.\] Thus,  
\begin{equation*}
\|u_h-\Phi_{\lambda_{h},c_{h},X_{h}}-A_T\|^2_{L^2(Q_T)}\gtrsim h^4 c^{-(N-2s+4)}_{h} \|f\|^2_{L^2(T)}.
\end{equation*}
Summing over $T\in\mathcal{T}_{h,1}$ and using \eqref{eq:25} we get:
 \begin{equation*}
 \begin{split}
I_{2h} &= \int_{\BB_h}\int_{\BB_h} \frac{|(u_h-\Phi_{\lambda_{h},c_{h},X_{h}})(x)-(u_h-\Phi_{\lambda_{h},c_{h},X_{h}})(y)|^2}{|x-y|^{N+2s}} \, \dd x\, \dd y\\
 & \gtrsim h^{4-2s}c^{-(N-2s+4)}_{h} \sum_{T\in\mathcal{T}_{h,1}} \|f\|^2_{L^2(T)}\\
 & \gtrsim h^{4-2s}c^{-(N-2s+4)}_{h} \int_{2M c_h\leq |x-X_{h}|,|x|<1-h} \left(1+\frac{|x-X_{h}|^2}{c^2_{h}}\right)^{-(N-2s+2)} \, \dd x.
 \end{split}
 \end{equation*}

If $|X_{h}|<\frac{1}{2}$, then the set $\{x\in\RR^N: 2M c_h\leq |x-X_{h}|<\frac{1}{4}\}$ is included in $\omega_h\coloneqq\{x\in \RR^N: 2M c_h\leq |x-X_{h}|, |x|<1-h\}\subset \BB_h$, since $|x|\leq |X_{h}|+|x-X_{h}|<\frac{1}{2}+\frac{1}{4}<1-h$. This implies that 
\begin{align*}
   I_{2h}
     & \gtrsim h^{4-2s} c_{h}^{-(N-2s+4)} \int_{2M c_h\leq |x-X_{h}|<\frac{1}{4}} \left(1+\frac{|x-X_{h}|^2}{c_{h}^2}\right)^{-(N-2s+2)} \, \dd x\\
     & \gtrsim h^{4-2s} c_{h}^{-(N-2s+4)} \int_{2M c_h}^{\frac{1}{4}} \sigma^{N-1} \left(1+\frac{\sigma^2}{c_{h}^2}\right)^{-(N-2s+2)}\, \dd \sigma \\
 & \gtrsim h^{4-2s} c_{h}^{-(4-2s)} \int_{2M}^{\frac{1}{4c_{h}}} \xi^{N-1}(1+\xi^2)^{-(N-2s+2)} \, \dd \xi\\
  & \gtrsim \frac{h^{4-2s}}{c_{h}^{4-2s}} \int_{2M}^{4M} \xi^{N-1}(1+\xi^2)^{-(N-2s+2)} \, \dd \xi\gtrsim \left(\frac{h}{c_{h}}\right)^{4-2s},
 \end{align*}
where we have used that $\frac{1}{4c_h}\geq 4M$, by \eqref{eq:choice-of-h-step4}. Note also that $2Mc_h\leq \frac{1}{8}<\frac14$, so the set $\{x\in\RR^N: 2M c_h\leq |x-X_{h}|<\frac{1}{4}\}$ is not empty.

If $|X_{h}|\geq \frac{1}{2}$, we will treat two cases: when $N=1$ and when $N\geq 2$.\\
\textit{Case 2.a.} $N=1$:\\
Without loss of generality, we assume that $\frac{1}{2}\leq X_h\leq 1$. We obtain that $$\left\{x\in\RR: 2M c_h\leq X_{h}-x\leq \frac{1}{2}\right\}\subset \omega_h.$$
Indeed, since $M\geq C_1\geq \frac{h}{c_h}$, we get that, 
$$0\leq X_h-\frac{1}{2}\leq x\leq X_h-2Mc_h\leq 1-2h<1-h.$$
Denoting $$\omega'_h\coloneqq \left\{y\in \RR: 2Mc_h\leq y\leq \frac{1}{2}\right\},$$ we get:
\begin{align*}
   I_{2h}&\gtrsim h^{4-2s} c_{h}^{-(N-2s+4)} \int_{x\in \omega'_{h}} \left(1+\frac{x^2}{c^2_{h}}\right)^{-(N-2s+2)} \, \dd x\\
    & \gtrsim h^{4-2s} c_{h}^{-(4-2s)} \int_{y\in \frac{1}{c_{h}}\omega'_{h}} (1+y^2)^{-(N-2s+2)} \, \dd y \gtrsim \left(\frac{h}{c_{h}}\right)^{4-2s},
     \end{align*}
where we have used the fact that the set $\frac{1}{c_{h}} \omega'_h$ contains a set of positive measure, measure independent of $h$,
\begin{equation}
\Big\{y\in\RR: 2M\leq y\leq 4M\Big\} \subset \Big\{y\in\RR: 2M\leq y\leq \frac{1}{2c_h}\Big\}=\frac{1}{c_{h}}\omega'_h,
\end{equation}
where the inclusion holds since $\frac{1}{2c_h}\geq 8M$, by \eqref{eq:choice-of-h-step4}.\\
\textit{Case 2.b.} $N\geq 2$\\
In this case, even in the worst-case scenario (note that $X_h$ can be close to the boundary of $\BB_h$), we can always construct a cone centered at $X_{h}$ that is entirely contained in the set $\omega_h$. Without loss of generality, assume that $X_{h}=(x_{1,h},0')\in \RR\times \RR^{N-1}$ with $\frac{1}{2}\leq x_{1,h}\leq 1$. For $\beta<\frac{1}{2}$ we get that $$\left\{x=(x_1,x')\in\RR\times \RR^{N-1}: |x'|\leq \beta |x_{1,h}-x_{1}|, 2M c_h\leq x_{1,h}-x_{1}\leq \frac{1}{2}\right\}\subset \omega_h.$$ Indeed, in this case,  
\begin{align*}
   |x|&\leq |x_1|+\beta(x_{1,h}-x_{1})=\beta x_{1,h}+(1-\beta) x_1\\
   &\leq \beta x_{1,h}+(1-\beta)(x_{1,h}-2M c_h)=x_{1,h}-2M c_h (1-\beta)<1-h. 
\end{align*} 
Denoting $$\omega'_{h}\coloneqq\Big\{y=(y_1,y')\in\RR\times \RR^{N-1}: |y'|\leq \beta |y_1|, 2M c_h\leq y_1\leq \frac{1}{2}\Big\},$$ we obtain, after a change of variables, that 

\begin{align*}
   I_{2h}&\gtrsim h^{4-2s} c_{h}^{-(N-2s+4)} \int_{x\in \omega'_{h}} \left(1+\frac{|x|^2}{c^2_{h}}\right)^{-(N-2s+2)} \, \dd x\\
    & \gtrsim h^{4-2s} c_{h}^{-(4-2s)} \int_{y\in \frac{1}{c_{h}}\omega'_{h}} (1+|y|^2)^{-(N-2s+2)} \, \dd y \gtrsim \left(\frac{h}{c_{h}}\right)^{4-2s},
     \end{align*}
where we have used the fact that the set $\frac{1}{c_{h}} \omega'_h$ contains a set of positive measure, measure independent of $h$,
\begin{equation}
\begin{split}
\Big\{y=(y_1,y')&\in\RR\times \RR^{N-1}: |y'|\leq \beta |y_1|, 2M\leq y_1\leq 4M\Big\}\\
& \subset \Big\{y=(y_1,y')\in\RR\times\RR^{N-1}:|y'|\leq \beta |y_1|, 2M \leq y_1 \leq \frac{1}{2c_{h}}\Big\} \subset \frac{1}{c_{h}}\omega'_h,
\end{split}
\end{equation}
where the first inclusion holds since $\frac{1}{2c_h}\geq 8M$, by \eqref{eq:choice-of-h-step4}.
This finishes the proof of estimate \eqref{eq:17}.

{\begin{remark}
\label{rem:rate-ch}
Let $\gamma\coloneqq\frac{2(2-s)}{N+4(1-s)}$, so that $(N-2s)\gamma=\alpha$ and $2(2-s)(1-\gamma)=\frac{2(2-s)(N-2s)}{N+4(1-s)}=\alpha$.
\begin{enumerate}
\item The last inequality in the display following \eqref{eq:17} holds because, for every $c>0$,
\[
c^{N-2s}+\Big(\frac{h}{c}\Big)^{2(2-s)}\geq h^{\alpha}.
\]
Indeed, if $c\geq h^\gamma$, then $c^{N-2s}\geq h^{(N-2s)\gamma}=h^\alpha$. If $c<h^\gamma$, then $\big(\frac{h}{c}\big)^{2(2-s)}> h^{2(2-s)(1-\gamma)}=h^\alpha$.
\item The concentration scale $c_h$ of the bubble $\Phi_{\lambda_h,c_h,X_h}$ closest to $u_h$ satisfies $c_h\sim h^{\gamma}$. Indeed, by \eqref{eq:lowerbound}, \eqref{eq:26} and \eqref{eq:17},
\[
c_h^{N-2s}+\Big(\frac{h}{c_h}\Big)^{2(2-s)}\lesssim I_{1h}+I_{2h}\lesssim [u_h-\Phi_{\lambda_h,c_h,X_h}]^2_{\dot H^s(\RR^N)}=d(u_h,\mathcal{M})^2\lesssim h^{\alpha}.
\]
Hence $c_h\lesssim h^{\frac{\alpha}{N-2s}}=h^{\gamma}$ and $c_h\gtrsim h^{1-\frac{\alpha}{2(2-s)}}=h^{\gamma}$.
\end{enumerate}
\end{remark}
}

\appendix

\section{Classical inequalities} We recall here some classical inequalities that we used along the paper. Even though they are classical for completeness we prefer to include them. 

The first inequality is the Homogeneous Gagliardo-Nirenberg Interpolation Inequality  \cite{JVSchaftingen2023}, \cite[Theorems 7.41, 8.29]{GLeoni2023}.
\begin{theorem}
\label{thm:IntIneq}
 Let $\Omega\subseteq \RR^N$ be an open and convex set and $0\leq s_0<s<s_1\leq 1$, $1\leq p,k,q\leq\infty$ if $\Omega=\RR^N$ or else $1<p,k,q<\infty$, and $0<\theta<1$ be such that $$s=(1-\theta) s_0+\theta s_1$$ and $$\frac{1}{k}=\frac{1-\theta}{q} + \frac{\theta}{p}.$$ 
 If $\Omega=\RR^N$ we also assume that at least one of the following statements is true:
 \begin{enumerate} 
    \item $s_1\notin \mathbb{Z}$;
     \item $p>1$;
     \item \label{item:s1-s2-big} $s_1-s_0>1-\frac{1}{q}$.
 \end{enumerate}
 For all $u\in \dot{W}^{s_0,q}(\Omega)\cap \dot{W}^{s_1,p}(\Omega)$ we have that $u\in \dot{W}^{s,k}(\Omega)$ and 
 \begin{equation}\label{eq:InterpolationIneq}
      \|u\|_{\dot{W}^{s,k}(\Omega)}\lesssim \|u\|_{\dot{W}^{s_0,q}(\Omega)}^{1-\theta} \|u\|_{\dot{W}^{s_1,p}(\Omega)}^\theta.
 \end{equation}
\end{theorem}

\begin{remark}
\label{rem:GN-convex}
The underlying constant in \eqref{eq:InterpolationIneq} does not depend on the open and convex set $\Omega$.
\end{remark}

The second inequality is Poincar\'{e}'s inequality for fractional Sobolev spaces \cite[Theorem 6.33]{GLeoni2023}. 

\begin{theorem}[Poincar\'{e}'s Inequality]
\label{thm:Poincareinequality}
Let $\Omega\subset \RR^N$ be an open bounded set, let $E\subseteq\Omega$ be a Lebesgue measurable set with positive measure and let $1\leq p<\infty$ and $0<s<1$. Then for all $u\in W^{s,p}(\Omega)$,
\begin{equation}
    \int_{\Omega} |u(x)-u_E|^p \, \dd x\leq \frac{(\text{diam } \Omega)^{N+sp}}{|E|} \int_{\Omega}\int_{\Omega} \frac{|u(x)-u(y)|^p}{|x-y|^{N+sp}} \, \dd x\, \dd y,
\end{equation}
where $u_E$ denotes the mean of the function $u$ on the set $E$ i.e. $u_E=\frac{1}{|E|} \int_{E} u(x) \, \dd x$.
\end{theorem}

\section{Technical lemmas}
\begin{lemma}\label{lem:estimateforlambda} Let  $c\in \left(0,\frac{1}{3}\right)$ and $\lambda_c\in \RR\setminus\{0\}$ be such that \[\Psi_{\lambda_c,c,0}(x)=\Phi_{\lambda_c,c,0}(x)-\lambda_c\left(1+\frac{1}{c^2}\right)^{-\frac{N-2s}{2}}\] satisfies $\|\Psi_{\lambda_c,c,0}\|_{L^{2_s^*}(\BB)}=1$. Then $|\lambda_c|\sim c^{-\frac{N-2s}{2}}$.
\begin{proof}
Using the fact that $\|\Psi_{\lambda_c,c,0}\|_{L^{2_s^*}(\BB)}=1$ we get
\begin{equation*}
\begin{split}
&\int_{\BB} \left[\left(1+\frac{|x|^2}{c^2}\right)^{-\frac{N-2s}{2}}-\left(1+\frac{1}{c^2}\right)^{-\frac{N-2s}{2}}\right]^{\frac{2N}{N-2s}} \, \dd x=\frac{1}{|\lambda_c|^{\frac{2N}{N-2s}}}\end{split}
\end{equation*}
which leads to:
\begin{equation}\label{eq:radial-sim-lambdac}
\begin{split}
&\int_{0}^{1} r^{N-1} \left[\left(1+\frac{r^2}{c^2}\right)^{-\frac{N-2s}{2}}-\left(1+\frac{1}{c^2}\right)^{-\frac{N-2s}{2}}\right]^{\frac{2N}{N-2s}}\, \dd r \sim |\lambda_c|^{-\frac{2N}{N-2s}}.
\end{split}
\end{equation}
Let $v(r)\coloneqq\left(1+\frac{r^2}{c^2}\right)^{-\frac{N-2s}{2}}-\left(1+\frac{1}{c^2}\right)^{-\frac{N-2s}{2}}$. We note that, since $c\in \left(0,\frac{1}{3}\right)$, we obtain for every $r\leq c$:
\[v(r)\geq 2^{-\frac{N-2s}{2}}- 10^{-\frac{N-2s}{2}}>0.\]
Therefore, one can estimate the left-hand side of \eqref{eq:radial-sim-lambdac} as follows:
\begin{equation}
\label{eq:radial-estimate-below}
\int_{0}^{1} r^{N-1} |v(r)|^{\frac{2N}{N-2s}}\dd r \geq \int_0^c r^{N-1} |v(r)|^{\frac{2N}{N-2s}} \dd r \gtrsim \int_0^c r^{N-1}\dd r \sim c^N.
\end{equation}

Since $|v(r)|^{\frac{2N}{N-2s}}<\left(1+\frac{r^2}{c^2}\right)^{-N}$, we also get the upper bound:
\begin{equation}\label{eq:radial-estimate-above}
\begin{aligned}
\int_{0}^{1} r^{N-1}|v(r)|^{\frac{2N}{N-2s}} \, \dd r \lesssim c^N \int_{0}^{\infty} \frac{r^{N-1}}{(1+r^2)^N} \, \dd r \lesssim c^N.
\end{aligned}
\end{equation}
By \eqref{eq:radial-sim-lambdac}, \eqref{eq:radial-estimate-below} and \eqref{eq:radial-estimate-above}, we get that  $|\lambda_c|\sim c^{-\frac{N-2s}{2}}$.
\end{proof}
\end{lemma}

\begin{lemma}[The gradient and the Hessian of the minimizers]\label{lem:GradHessMinimizers}
We have
\begin{equation}
|D\Phi_{\lambda,c,X_0}(x)|=|\lambda| \frac{N-2s}{c^2} \left(1+\frac{|x-X_0|^2}{c^2}\right)^{-\frac{N-2s+2}{2}}|x-X_0|
\end{equation}
and
\begin{equation}
\begin{split}
&|D^2\Phi_{\lambda,c,X_0}(x)|=|\lambda|\frac{N-2s}{c^2}\left(1+\frac{|x-X_0|^2}{c^2}\right)^{-\frac{N-2s+4}{2}}\times\\
&\quad \times\sqrt{\left(1-(N-2s+1)\frac{|x-X_0|^2}{c^2}\right)^2+(N-1)\left(1+\frac{|x-X_0|^2}{c^2}\right)^2}.  
\end{split}
\end{equation}
Moreover,
\begin{equation}\label{second.derivative}
|D^2\Phi_{\lambda,c,X_0}(x)|\lesssim \frac{|\lambda|}{c^2}\left(1+\frac{|x-X_0|^2}{c^2}\right)^{-\frac{N-2s+2}{2}}.
\end{equation}
\end{lemma}
\begin{proof} Let $u(r)\coloneqq\left(1+\frac{r^2}{c^2}\right)^{-\frac{N-2s}{2}}.$
Then
\begin{equation*} 
u'(r)=-\frac{N-2s}{c^2} \left(1+\frac{r^2}{c^2}\right)^{-\frac{N-2s+2}{2}} r
\end{equation*}
and 
\begin{equation*}
u''(r)=-\frac{N-2s}{c^2}\left(1+\frac{r^2}{c^2}\right)^{-\frac{N-2s+4}{2}}\left[1-(N-2s+1)\frac{r^2}{c^2}\right].
\end{equation*}
So the first equality easily follows. For the second one, observe that
\begin{equation*}
\begin{split}
\partial_{x_ix_j}\Phi_{\lambda,c,X_0}(x) & =\lambda\left(u''(|x-X_0|)\frac{(x_i-X_{0,i})(x_j-X_{0,j})}{|x-X_0|^2}\right)\\
    &\quad  + \lambda \left(u'(|x-X_0|)\frac{\delta_{ij}|x-X_0|^2-(x_i-X_{0,i})(x_j-X_{0,j})}{|x-X_0|^3}\right)
\end{split}
\end{equation*}
and the conclusion easily follows. \end{proof}

\begin{lemma}\label{lem:interpolatior-on-polyhedron-Lq-error}
For any  $q\in [1,\infty)$ and $h<<c_h<<1$ the following holds
\begin{equation}
\|\Psi_{\lambda_h,c_h,0}-I_h(\Psi_{\lambda_h,c_h,0})\|_{L^q(\BB_h)}\lesssim h^{2}c_h^{-\left(\frac{N}{2}-\frac{N}{q}+2-s\right)}.
\end{equation}
\end{lemma}
\begin{proof} We divide the proof in two steps.\\
\textbf{Case I: $q>\frac{N}{2}$.} In this case we use the classical result in \cite[Theorem 4.4.20]{BrennerScott} to obtain:
\begin{equation}
\label{eq:Lq-error-in-dotW2q}
    \|\Psi_{\lambda_h,c_h,0}-I_h(\Psi_{\lambda_h,c_h,0})\|^q_{L^q(\BB_h)}\lesssim h^{2q} \|D^2\Psi_{\lambda_h,c_h,0}\|^q_{L^q(\BB_h)}.
\end{equation}\\
In order to estimate the right-hand side above we employ estimate \eqref{second.derivative} and obtain:
\begin{equation*}
\label{eq:Lq-norm-of-D2Psi}
\begin{split}
\|D^2\Psi_{\lambda_h,c_h,0}\|^q_{L^q(\BB_h)}& =\|D^2\Phi_{\lambda_h,c_h,0}\|^q_{L^q(\BB_h)}\\
& \lesssim\frac{|\lambda_h|^q}{c^{2q}_h} \int_{\BB_h} \left(1+\frac{|x|^2}{c^2_h}\right)^{-\frac{N-2s+2}{2}q}\, \dd x\\
& \lesssim \frac{|\lambda_h|^q}{c^{2q}_h} \int_{0}^{1} \sigma^{N-1}\left(1+\frac{\sigma^2}{c^2_h}\right)^{-\frac{N-2s+2}{2}q} \, \dd \sigma\\
&\lesssim |\lambda_h|^q c_h^{N-2q},
\end{split}
\end{equation*}
which finishes the proof of this case.

\textbf{Case II: $q\leq \frac{N}{2}$}.
We proceed as in the proof of \cite[Lemma 3.10]{ignat2025optimalconvergenceratesfinite} and define the following partition of the set of mesh simplices $\mathcal{T}_h$:
    \begin{equation*}
    \begin{aligned}
        \TT_h^1&\coloneqq \{T\in \TT_h: T\cap B_{2h}(0)= \emptyset\};\\
        \TT_h^2&\coloneqq \{T\in \TT_h: T\cap  B_{2h}(0)\neq \emptyset  \}.
        \end{aligned}
    \end{equation*}
We note that, since the diameter of the simplices in $\TT_h$ is at most equal to $h$, one has $\TT_h^2\subseteq \TT_h^3 \coloneqq \{T\in \TT_h: T\subset \overline{B_{3h}(0)} \}.$ Therefore, 
\begin{equation*}
\begin{split}
    \|\Psi_{\lambda_h,c_h,0}-I_h&(\Psi_{\lambda_h,c_h,0})\|_{L^q(\BB_h)}^q\\
    &\leq\sum_{T\in \TT_h^1} \int_{T} |\Psi_{\lambda_h,c_h,0}-I_h(\Psi_{\lambda_h,c_h,0})|^q
    +\sum_{T\in \TT_h^3} \int_{T} |\Psi_{\lambda_h,c_h,0}-I_h(\Psi_{\lambda_h,c_h,0})|^q\\
&=: I_1+I_2.
\end{split}
\end{equation*}
Next, we use Lemma \ref{lem:interpolating-W2infty-functions} to write:
 \begin{equation}\label{eq:above-I1-estimate-0}
 I_1\lesssim_q h^{2q}\sum_{T\in \TT_h^1}|T|\,\|D^2 \Psi_{\lambda_h,c_h,0}\|_{L^{\infty}(T)}^{q}= h^{2q}\sum_{T\in \TT_h^1}|T|\,\|D^2 \Phi_{\lambda_h,c_h,0}\|_{L^{\infty}(T)}^{q}.
 \end{equation}
Using estimate \eqref{second.derivative}, we obtain   
      \begin{equation}
      \label{sup-norm-of-D2Phi-on-triangle}
      \begin{split}
      |T|\|D^2 \Phi_{\lambda_h,c_h,0}\|_{L^{\infty}(T)}^{q}& \lesssim  |T|\frac{|\lambda_h|^q}{c^{2q}_h} \sup_{y\in T} \left(1+\frac{|y|^2}{c_h^2}\right)^{-\frac{N-2s+2}{2}q}
      \end{split}\end{equation}

     We aim to return to the $L^q(T)$ setting by proving that the norms of points in a simplex $T\in\TT_h^1$ are comparable with each other. Indeed, for every $x,y\in T\in \TT_h^1$, we have,
      $$|x|\geq 2h\geq 2 h_T\geq 2|x-y|\geq 2|x|-2|y|$$ 
      and, as a result, $2|y|\geq |x|$. It follows that:
      $$\sup_{y\in T} |y|\leq 2\inf_{y\in T} |y|,$$
      which further implies that:
      \[\sup_{y\in T} \left(1+\frac{|y|^2}{c_h^2}\right)^{-\frac{N-2s+2}{2}q}\leq 4^{\frac{N-2s+2}{2}q} \inf_{y\in T} \left(1+\frac{|y|^2}{c_h^2}\right)^{-\frac{N-2s+2}{2}q}\leq C_{N,s,q} \frac{1}{|T|} \int_T \left(1+\frac{|y|^2}{c_h^2}\right)^{-\frac{N-2s+2}{2}q} \, \dd y\]
      where $C_{N,s,q}$ is a positive constant which depends only on $N, s, q$.
      Plugging this inequality into \eqref{eq:above-I1-estimate-0} and \eqref{sup-norm-of-D2Phi-on-triangle}, we obtain that:
\begin{equation}\label{eq:above-I1-estimate-1}
\begin{split}
I_1&\lesssim h^{2q}\frac{|\lambda_h|^q}{c_h^{2q}}\sum_{T\in \TT_h^1}\int_{T} \left(1+\frac{|x|^2}{c_h^2}\right)^{-\frac{N-2s+2}{2}q}\, \dd x\\
&\lesssim h^{2q}\frac{|\lambda_h|^q}{c_h^{2q}} \int_{ \{ |x|\geq 2h \} \cap \BB_h } \left(1+\frac{|x|^2}{c_h^2}\right)^{-\frac{N-2s+2}{2}q}\, \dd x \\
&\lesssim h^{2q} c_h^{\frac{N-2s}{2}q}\int_{0}^{1} \frac{\sigma^{N-1}}{(c_h^2+\sigma^2)^{\frac{N-2s+2}{2}q}}\, \dd\sigma \sim  h^{2q} c_h^{-\left(\frac{Nq}{2}-N+2q-sq\right)}.\end{split}
\end{equation}

For estimating $I_2$ we use the fact that $0\leq \Psi_{\lambda_h,c_h,0}(x)\leq \Phi_{\lambda_h,c_h,0}(x)$ for any $x\in \overline{\BB}$ to obtain:
\begin{equation}\label{eq:estimate-I-2}
    \begin{split}
    I_2 &\lesssim \sum_{T\in \TT_h^3} \left[\int_{T} |\Psi_{\lambda_h,c_h,0}(x)|^q +\int_{T} |I_h(\Psi_{\lambda_h,c_h,0})(x)|^q \, \dd x \right]\\
    &\lesssim  \int_{\overline{B}_{3h}(0)} |\Psi_{\lambda_h,c_h,0}(x)|^q \, \dd x+\int_{\overline{B}_{3h}(0)} |I_h(\Psi_{\lambda_h,c_h,0})(x)|^q \, \dd x\\
    &\lesssim \int_{\overline{B}_{3h}(0)} |\Psi_{\lambda_h,c_h,0}(x)|^q \, \dd x+h^N\|I_h(\Psi_{\lambda_h,c_h,0})\|_{L^{\infty}(\overline{B}_{3h}(0))}^q\\
    &\lesssim \int_{\overline{B}_{3h}(0)} |\Psi_{\lambda_h,c_h,0}(x)|^q \, \dd x+h^N\|\Psi_{\lambda_h,c_h,0}\|_{L^{\infty}(\overline{B}_{3h}(0))}^q\\
    & \lesssim c_h^{\frac{N-2s}{2}q} \int_{0}^{3h}   \frac{\sigma^{N-1}}{(c_h^2+\sigma^2)^{\frac{N-2s}{2}q}}  \, \dd\sigma +   \frac{h^N}{c_h^{\frac{N-2s}{2}q}} \\
    & \lesssim c_h^{\frac{N-2s}{2}q} \int_{0}^{3h} \frac{\sigma^{N-1}}{(c_h^2)^{\frac{N-2s}{2}q}} \, \dd\sigma + \frac{h^N}{c_h^{\frac{N-2s}{2}q}}\\
    & \lesssim \frac{h^N}{c_h^{\frac{N-2s}{2}q}}\lesssim \frac{h^{2q}}{c_h^{\frac{Nq}{2}-N+2q-sq}}.
    \end{split}
\end{equation}
Putting together the results in \eqref{eq:above-I1-estimate-1} and \eqref{eq:estimate-I-2} we obtain the desired result.
\end{proof}

\begin{lemma}
\label{lem:gradient-of-interpolatior-error-on-polyhedron}
For any $p\in [1,\infty)$ and $h<<c_h<<1$  the following holds:
\begin{equation}
\|D\left(\Psi_{\lambda_h,c_h,0}-I_h(\Psi_{\lambda_h,c_h,0})\right)\|_{L^p(\BB_h)}\lesssim hc_h^{- \left(\frac{N}{2}-\frac{N}{p}+2-s\right)}.
\end{equation}
\end{lemma}
\begin{proof}
We proceed as in the proof of Lemma \ref{lem:interpolatior-on-polyhedron-Lq-error}.
When $p>\frac{N}{2}$ we use \cite[Theorem 4.4.20]{BrennerScott} to get:
\begin{equation*}
    \|D(\Psi_{\lambda_h,c_h,0}-I_h(\Psi_{\lambda_h,c_h,0}))\|_{L^p(\BB_h)}^p\lesssim_p h^p \|D^2\Psi_{\lambda_h,c_h,0}\|^p_{L^{p}(\BB_h)}\lesssim_p h^p c_h^{N-\frac{N-2s+4}{2}p}.
\end{equation*}
For $p\leq \frac{N}{2}$ we obtain:
\begin{equation}\label{eq:L-p-norm-of-the-gradient-of-the-interpolation-error}
\begin{split}
    \int_{\BB_h} |D(\Psi_{\lambda_h,c_h,0}-I_h(\Psi_{\lambda_h,c_h,0}))|^p &\leq \sum_{T\in \TT_h,T\subset \{|x|\geq 2h\}} \int_{T} |D(\Psi_{\lambda_h,c_h,0}-I_h(\Psi_{\lambda_h,c_h,0}))|^p\\
    &\quad +\sum_{T\in \TT_h, T\subset \{|x|\leq 3h\}} \int_{T} |D(\Psi_{\lambda_h,c_h,0}-I_h(\Psi_{\lambda_h,c_h,0}))|^p\\
&=:\tilde{I}_1+\tilde{I}_2.
\end{split}
\end{equation}
Using Lemma \ref{lem:interpolating-W2infty-functions} and Lemma \ref{lem:GradHessMinimizers} we get that
\begin{equation}\label{eq:Estimate-tilde-I-1}
    \tilde{I_1}\lesssim h^p\int_{\BB_h \cap \{|x|\geq 2h\}} |D^2\Psi_{\lambda_h,c_h,0}(x)|^p\, \dd x \lesssim_p \frac{h^p}{c_h^{\frac{Np}{2}-N+2p-sp}}
\end{equation}
and 
\begin{equation}\label{eq:estimate-on-tilde-I-2}
\begin{split}
\tilde{I_2}&\lesssim \sum_{T\in \TT_h, T\subset \{|x|\leq 3h\}} \left(\int_{T} |D\Psi_{\lambda_h,c_h,0}(x)|^p \, \dd x+\int_{T} |D I_h(\Psi_{\lambda_h,c_h,0})(x)|^p \, \dd x \right)\\
&\lesssim \int_{|x|\leq 3h} |D\Psi_{\lambda_h,c_h,0}(x)|^p \, \dd x+h^{N-p} \|\Psi_{\lambda_h,c_h,0}\|_{L^\infty(\{|x|\leq 3h\})}^p\\
&\lesssim c_h^{\frac{N-2s}{2}p}\int_{0}^{3h} \frac{\sigma^{p+N-1}}{(c_h^2+\sigma^2)^{\frac{N-2s+2}{2}p}} \, \dd\sigma+\frac{h^{N-p}}{c_h^{\frac{N-2s}{2}p}}\\
& \lesssim c_h^{\frac{N-2s}{2}p} \int_{0}^{3h} \frac{\sigma^{p+N-1}}{(c_h^2)^{\frac{N-2s+2}{2}p}} \, \dd \sigma + \frac{h^{N-p}}{c_h^{\frac{N-2s}{2}p}}\\
& \lesssim \frac{h^{N+p}}{c_h^{\frac{N-2s+4}{2}p}}+\frac{h^{N-p}}{c_h^{\frac{N-2s}{2}p}}\lesssim \frac{h^p}{c_h^{\frac{Np}{2}-N+2p-sp}}.
\end{split}
\end{equation}
By \eqref{eq:L-p-norm-of-the-gradient-of-the-interpolation-error}, \eqref{eq:Estimate-tilde-I-1} and \eqref{eq:estimate-on-tilde-I-2} we obtain the desired estimate.
\end{proof}

{\begin{lemma}
\label{lem:finitecovering}
Let $N=1$, $s\in \left(0,\frac{1}{2}\right)$, $\lambda\in\RR\setminus\{0\}$, $c>0$ and $X_0\in\RR$. Then
\[
|D^2\Phi_{\lambda,c,X_0}(x)|\geq A_{s} \frac{|\lambda|}{c^2} \left(1+\frac{|x-X_0|^2}{c^2}\right)^{-\frac{N-2s+2}{2}}
\]
whenever $|x-X_0|\leq \frac{c}{2}$ or $|x-X_0|\geq c$, where $A_s\coloneqq (1-2s)\min\left\{\frac{2}{5},\frac{1-2s}{2}\right\}$.
\end{lemma}
\begin{proof}
Let $t\coloneqq \frac{|x-X_0|}{c}$. By Lemma \ref{lem:GradHessMinimizers} with $N=1$,
\[
|D^2\Phi_{\lambda,c,X_0}(x)|=(1-2s)\frac{|\lambda|}{c^2}(1+t^2)^{-\frac{5-2s}{2}}\left|1-(2-2s)t^2\right|.
\]
If $t\leq\frac12$, then $|1-(2-2s)t^2|\geq 1-\frac{2-2s}{4}\geq\frac12$ and $(1+t^2)^{-1}\geq\frac45$. Hence $|D^2\Phi_{\lambda,c,X_0}(x)|\geq \frac{2(1-2s)}{5}\frac{|\lambda|}{c^2}(1+t^2)^{-\frac{3-2s}{2}}$.
If $t\geq 1$, then $(2-2s)t^2-1\geq (1-2s)t^2$ and $\frac{t^2}{1+t^2}\geq\frac12$. Hence $|D^2\Phi_{\lambda,c,X_0}(x)|\geq \frac{(1-2s)^2}{2}\frac{|\lambda|}{c^2}(1+t^2)^{-\frac{3-2s}{2}}$.
\end{proof}

\begin{lemma}
\label{lem:orthogonal-direction}
Let $N\geq 2$, $s\in(0,1)$, $\lambda\in\RR\setminus\{0\}$, $c>0$ and $X_0\in\RR^N$. For every $x\in\RR^N$ and every $\xi\in\mathbb{S}^{N-1}$ we have
\begin{equation}
\label{eq:xidirection}
\begin{split}
\xi^T D^2\Phi_{\lambda,c,X_0}(x) \xi=&-(N-2s)\frac{\lambda}{c^2}\left(1+\frac{|x-X_0|^2}{c^2}\right)^{-\frac{N-2s+4}{2}}\\
&\times\left\{1+\frac{|x-X_0|^2}{c^2}-(N-2s+2)\frac{((x-X_0)\cdot\xi)^2}{c^2}\right\}.
\end{split}
\end{equation}
Moreover, let $B_r(y_0)$ be a ball such that
\begin{equation}
\label{eq:cond-orth}
2(N+2)\,r^2\leq c^2+\operatorname{dist}(X_0,B_r(y_0))^2,
\end{equation}
and let $\xi\in\mathbb{S}^{N-1}$ be orthogonal to $y_0-X_0$ (if $y_0=X_0$, then $\xi$ is arbitrary). Then
\begin{equation}
\label{eq:orth-lower}
|\xi^T D^2\Phi_{\lambda,c,X_0}(x)\xi|\geq \frac{N-2s}{2}\,\frac{|\lambda|}{c^2}\left(1+\frac{|x-X_0|^2}{c^2}\right)^{-\frac{N-2s+2}{2}}\qquad \text{for all } x\in \overline{B_r(y_0)}.
\end{equation}
In particular,
\[
\max_{\xi\in\mathbb{S}^{N-1}} \min_{x\in \overline{B_r(y_0)}} |\xi^T D^2\Phi_{\lambda,c,X_0}(x)\xi|\geq \frac{N-2s}{2}\,\frac{|\lambda|}{c^2} \min_{x\in\overline{B_r(y_0)}}\left(1+\frac{|x-X_0|^2}{c^2}\right)^{-\frac{N-2s+2}{2}}.
\]
\end{lemma}
\begin{proof}
Let $u$ be as in the proof of Lemma \ref{lem:GradHessMinimizers}. For $x\neq X_0$, set $\varrho\coloneqq|x-X_0|$ and $\omega\coloneqq (x-X_0)/\varrho$. The expression of $\partial_{x_ix_j}\Phi_{\lambda,c,X_0}$ in that proof gives
\[
\xi^T D^2\Phi_{\lambda,c,X_0}(x) \xi=\lambda\left[u''(\varrho)(\omega\cdot\xi)^2+\frac{u'(\varrho)}{\varrho}\left(1-(\omega\cdot\xi)^2\right)\right].
\]
Inserting the formulas for $u'$ and $u''$ and using $\varrho^2(\omega\cdot\xi)^2=((x-X_0)\cdot\xi)^2$, we obtain \eqref{eq:xidirection}. For $x=X_0$, \eqref{eq:xidirection} follows by continuity.

Let now $x\in\overline{B_r(y_0)}$. Since $\xi\perp y_0-X_0$, we have $((x-X_0)\cdot\xi)^2=((x-y_0)\cdot\xi)^2\leq r^2$. Moreover, $N-2s+2\leq N+2$ and $|x-X_0|\geq \operatorname{dist}(X_0,B_r(y_0))$. By \eqref{eq:cond-orth}, the term in braces in \eqref{eq:xidirection} satisfies
\[
\begin{split}
1+\frac{|x-X_0|^2}{c^2}-(N-2s+2)\frac{((x-X_0)\cdot\xi)^2}{c^2}&\geq 1+\frac{|x-X_0|^2}{c^2}-\frac{(N+2)r^2}{c^2}\\
&\geq 1+\frac{|x-X_0|^2}{c^2}-\frac{c^2+|x-X_0|^2}{2c^2}\\
&=\frac12\left(1+\frac{|x-X_0|^2}{c^2}\right).
\end{split}
\]
Inserting this into \eqref{eq:xidirection} gives \eqref{eq:orth-lower}.
\end{proof}
}

\begin{lemma}
\label{lem:interpolationoncubes}
For any cube $Q\subset \RR^N$ with side-length $l$ and any $1\leq p<\infty$ there exists a constant $c_{N,p}>0$ such that the following inequality holds for any function $u\in L^p(Q)\cap \dot{W}^{2,p}(Q)$:
\begin{equation}
\|Du\|_{L^p(Q)}\leq \frac{c_{N,p}}{l} \|u\|_{L^p(Q)}+c_{N,p}\|u\|^{\frac{1}{2}}_{L^p(Q)} \|D^2 u\|^{\frac{1}{2}}_{L^p(Q)}.
\end{equation}
\end{lemma}
\begin{proof}
The proof uses \cite[Theorem 13.51]{GLeoni2017} and the fact that, for any real number $p>0$ and any nonnegative real numbers $a_j$, $1\leq j\leq n$, $\sum_{j=1}^{n} a^p_j\sim_{n,p} (\sum_{j=1}^{n} a_j)^p$.
\end{proof}

\subsection*{Acknowledgements}
We would like to thank Drago\c{s} Manea for various discussions when this project started and the anonymous referee for their valuable suggestions. A. Dima was partially supported by a scholarship of the SCOSAAR. 
L. I.  Ignat  was partially supported by a grant of the Ministry of Research, Innovation, and Digitization, CCCDI -
UEFISCDI, project number ROSUA-2024-0001, within PNCDI IV. This article is based upon work from COST Action 24122 mSPACE, supported by COST (European Cooperation in Science and Technology), www.cost.eu.

\subsection*{AI Disclosure}
The authors used Claude Opus 5.5 in order to improve the first version of the manuscript. After using this tool, the authors reviewed and edited the content as needed. The authors take full responsibility for the content of the publication.

\printbibliography
\end{document}